\documentclass[a4paper, 10pt, reqno]{amsart}
\usepackage[margin=1in]{geometry}
\usepackage{amsmath}
\usepackage{amssymb}
\usepackage{amsthm}
\usepackage{mathtools}
\usepackage{tikz}
\usepackage{bbm}
\usepackage{dsfont}
\usepackage{appendix}

\newtheorem{thm}{Theorem}
\newtheorem*{thm*}{Theorem}
\newtheorem*{conj*}{Conjecture}
\newtheorem{lemma}{Lemma}
\newtheorem*{lemma*}{Lemma}
\newtheorem{prop}{Proposition}

\newtheorem*{remarks*}{Remarks}
\title{
Effective Results in The Metric Theory of Quantitative Diophantine Approximation}
\author{Ying Wai Lee \& Andrew Scoones}
\date{\today}

\begin{document}

\maketitle
\begin{abstract}
    Many results related to quantitative problems in the metric theory of Diophantine approximation are asymptotic, such as the number of rational solutions to certain inequalities grows with the same rate almost everywhere modulo an asymptotic error term. 
    The error term incorporates an implicit constant that varies from one point to another. This means that applications of these results does not give concrete bounds when applied to, say a finite sum, or when applied to counting the number of solutions up to a finite point for a given inequality. This paper addresses this problem and makes the tools and their results effective, by making the implicit constant explicit outside of an exceptional subset of Lebesgue measure at most $\delta>0$, an arbitrarily small constant chosen in advance. We deduce from this the fully effective results for Schmidt's Theorem, quantitative Koukoulopoulos-Maynard Theorem and quantitative results on $M_{0}$-sets; we also provide effective results regarding statistics of normal numbers and strong law of large numbers.
\end{abstract}


\section{Introduction}





Roughly speaking, we can think of Diophantine approximation as a branch of number theory that studies quantitatively the density of the rationals within the reals. This area is one of the main topics in Metric Number Theory, where we are often interested in studying the size (in terms of, for example, the Lebesgue measure or Hausdorff dimension) of Diophantine sets satisfying certain properties. One problem here is, given an approximating function $\psi: \mathbb{N} \to [0,\, \infty)$, to determine the size of the set
\[W(\psi)\coloneqq\limsup_{q\to\infty}\left\{x \in [0,1)\,:\,\left|x-\frac{p}{q}\right|<\frac{\psi(q)}{q}\textrm{, for some }p\in\mathbb{N}\right\}.\]
We say that the elements of $W(\psi)$ are $\psi$-approximable; and note that $x \in W(\psi)$ if there exist infinitely many positive integers $(p,\,q)$ satisfying 
\begin{equation}\label{psi approximable}
\left|x-\frac{p}{q}\right|<\frac{\psi(q)}{q}.    
\end{equation}

It is clear that every $x \in \mathbb{R}$ is within $\frac{1}{2q}$ of a rational number with denominator $q$; that is, we can take $\psi(q)=\frac{1}{2}$ for all $q$ and it is clear that $\mu\left(W\left(\psi\right)\right)=1$, where $\mu$ denotes the Lebesgue measure on $\mathbb{R}$. It is natural to ask how far this can be improved.

We now consider $\psi(q)=1/q$; in this case, the theory of continued fractions tells us that every $x$ is $\frac{1}{q}$-approximable. In fact, this is the content of a theorem of Dirichlet.

\begin{thm*}[Dirichlet's Theorem, 1842]
For any $x \in \mathbb{R}$ and $N \in \mathbb{N}$, there exist $p,\,q \in \mathbb{Z}$ such that
\[\left|x-\frac{p}{q}\right|<\frac{1}{qN},\]
with $1 \leq q \leq N$.
\end{thm*}
This has the following immediate corollary; we note the proof of the following can also be given via continued fractions, the theory of which is older than Dirichlet's Theorem.
\begin{thm*}
Let $x \in \mathbb{R}\setminus\mathbb{Q}$. Then there exist infinitely many $p,\,q$ such that $\gcd(p,\,q)=1$ and
\begin{equation}\label{Dirichlet}
\left| x-\frac{p}{q}\right| < \frac{1}{q^{2}}.    
\end{equation}
\end{thm*}

We note that this theorem is true for all $x \in \mathbb{R}$ if we remove the condition that $\gcd(p,\,q)=1$. Further, to connect this with our discussion above, we note we have shown that $\mu\left(W\left(\frac{1}{q}\right)\right)=1$; we recall in our definition of $W(\psi)$ we do not insist that $\gcd\left(p,\,q\right)=1$.

Now it is natural to ask how much we can improve inequality \eqref{Dirichlet}, and a theorem by Hurwitz shows that we can only improve this inequality for all $x \in \mathbb{R}$ so far.

\begin{thm*}[Hurwitz's Theorem, 1891]
Let $x \in \mathbb{R}\setminus\mathbb{Q}$. Then there exist infinitely many $p,\,q$ such that $\gcd(p,\,q)=1$ and
\begin{equation}\
\left|x-\frac{p}{q}\right| < \frac{1}{\sqrt{5}q^{2}}. 
\end{equation}
\end{thm*}

We note that the constant $1/\sqrt{5}$ is the best constant possible; for all $\varepsilon>0$, we can find an irrational $x$ such that 
\begin{equation}\label{Hurwitz}
 \left| x-\frac{p}{q}\right| < \frac{1}{(\sqrt{5}+\varepsilon)q^{2}}  
\end{equation}
has only finitely many coprime solution pairs $(p,\,q)$. Namely, if we take $x$ to be the Golden Ratio $\phi =\frac{1+\sqrt{5}}{2}$ then for any $\varepsilon>0$, we can only find finitely many integer solution pairs $(p,\,q)$ with $\gcd\left(p,\,q\right)=1$ to \eqref{Hurwitz}.

This comment means that if we take $\psi(q)<\frac{1}{\sqrt{5}q}$ then \eqref{psi approximable} will not hold for all $x\in \mathbb{R}$. This naturally leads us to ask under what conditions on $\psi$ can we say that $W(\psi)$ has full measure? Khintchine gave the following elegant theorem as an answer to this question:
\begin{thm*}[Khintchine, 1924]
 Let $\psi:\mathbb{N}\to[0,+\infty)$ be a function. Suppose $\psi$ is (eventually) non-increasing. Then
    \begin{equation*}
        \mu(W(\psi))= \begin{cases}
        0, &\textrm{ if } \sum_{q=1}^{\infty} \psi(q) < +\infty, \\
        1, &\textrm{ if } \sum_{q=1}^{\infty} \psi(q) = +\infty ,
        \end{cases}
    \end{equation*}
where $\mu$ is the Lebesgue measure on $[0,1)$.
\end{thm*}
One may wonder why $W(\psi)$ must have either full or zero measure; this is due to a zero-full law by Cassels’ zero-full law (see \cite{Cassels_1950}), this kind of phenomenon is typical for a large class of $\limsup$ sets (see \cite{harman1998metric} for further discussion).

Khintchine's Theorem tells us that, if $\psi$ is non-increasing, the convergence or divergence of the sum 
\begin{equation}\label{sum}
    \sum_{q=1}^{\infty} \psi(q)
\end{equation} 
entirely determines whether or not almost all $x \in [0,1)$ are $\psi$-approximable; for further details we refer the reader to \cite{harman1998metric}. We note that Khintchine's Theorem lets us give immediate improvements on the example $\psi(q)=\frac{1}{q}$ given above; for example, as
\[\sum_{q=1}^{\infty}\frac{1}{q \log(q+1)}=+\infty,\]
 we know that $W\left(\frac{1}{q \log(q+1)}\right)$ has full Lebesgue measure; that is we can improve the function given previously by a logarithmic factor.

One may ask whether we can remove the condition that $\psi$ is monotonic in Khintchine's Theorem. In the convergence case, the condition can be removed as the proof is an application of the Borel-Cantelli Lemma. The divergence case is somewhat more tricky, and monotonicity is required. In \cite{duffin1941khintchine}, Duffin and Schaeffer consider the function $\vartheta$ which is non-monotonic, such that $\sum_{q}\vartheta(q)$ diverges, but $\mu\left(W\left(\vartheta\right)\right)=0$; we refer the reader to \cite{duffin1941khintchine} for further information.

Duffin and Schaeffer went on to consider a generalisation of this problem, considering arbitrary approximation functions $\psi:\mathbb{N} \to \mathbb{R}^{+}$. We note that, contrary to Dirichlet's Theorem above, we do not have the condition that $\gcd(p,\,q)=1$ in the definition of $W(\psi)$. Duffin and Schaeffer also added this condition to relate the rational $\frac{p}{q}$ with the error of approximation $\frac{\psi(q)}{q}$ uniquely. They thus considered the set
\[W'(\psi)\coloneqq \limsup_{q\to\infty}{\left\{x\in[0,1):\left|x-\frac{p}{q}\right|<\frac{\psi(q)}{q}\textrm{, for some }p\in\mathbb{N},\,\gcd{(p,q)}=1.\right\}}.\]
They went on to give the following conjecture (previously known as the Duffin-Schaeffer Conjecture) which was proven by Koukoulopoulos and Maynard in 2019.

\begin{thm*}[Koukoulopoulos-Maynard, 2019]
Let $\psi:\mathbb{N}\to[0,+\infty)$ be a function. Then
    \begin{equation*}
        \mu(W'(\psi))= \begin{cases}
        0, &\textrm{ if } \sum_{q=1}^{\infty} \frac{\psi(q)\varphi(q)}{q} < +\infty, \\
        1, &\textrm{ if } \sum_{q=1}^{\infty} \frac{\psi(q)\varphi(q)}{q} = +\infty ,
        \end{cases}
    \end{equation*}
where $\mu$ is the Lebesgue measure on $[0,1)$ and $\varphi(x)$ is the Euler phi function, defined to be the number of natural numbers coprime to $x$ that are less than or equal to $x$.
\end{thm*}
That is, the convergence (or divergence respectively) of the series
\begin{align}\label{sum'}
    \sum_{q=1}^\infty\frac{\psi(q)\varphi(q)}{q}
\end{align}
implies that the Lebesgue measure of $W'(\psi)$ is 0 (or 1 respectively).

Again, the convergence case follows directly from the Borel-Cantelli Lemma, while the divergence case required deeper insight; for the full proof we refer to \cite{koukoulopoulos2020duffin}. We return to results around this theorem later in the thesis.

Khintchine's Theorem tells us about the measure of $W(\psi)$, but  to gain a deeper understanding of Khintchine's Theorem is to consider the number of solutions to \eqref{psi approximable}, given that $1 \leq q \leq Q$ for a fixed $Q \in \mathbb{N}$. We add the further restriction that $\psi: \mathbb{N} \to [0,1/2)$; this ensures that given any integer $q$, there is maximally one $p$ such that $(p,\,q)$ satisfy \eqref{psi approximable}. Thus, counting solution pairs $(p,\,q)$ to \eqref{psi approximable} is equivalent to counting the number of $q$'s for which we can find a $p$ satisfying \eqref{psi approximable}, subject to the condition that $1\leq q\leq Q$. More explicitly, given $x \in [0,1)$ and $Q \in \mathbb{N}$ we define
\begin{align*}
    S(x,Q)\coloneqq\#\left\{q\in\mathbb{N}\cap[1,Q]\,:\,\left|x-\frac{p}{q}\right|<\frac{\psi(q)}{q}\textrm{, for some }p\in\mathbb{N}\right\}; 
\end{align*} 
that is, $S(x,\,Q)$ is the number of positive integer pairs $(p,\,q)$ satisfying \eqref{psi approximable} and $1\leq q \leq Q$.

In this context, Khintchine's Theorem tells us that if $\psi$ is non-increasing and \eqref{sum} is divergent, then for almost every $x\in[0,1)$, 
\begin{align*}
    \lim_{Q\to\infty}S(x,Q)=+\infty. 
\end{align*}
We have seen that there are infinitely many solutions satisfying \eqref{psi approximable}, but Khintchine's Theorem does not give a growth rate in the number of solutions, nor does it suggest any quantitative relation between $S(x,\, Q)$ and $\psi$. We refer to these problems as the quantitative problem of Khintchine's Theorem. Schmidt was able to give asymptotic relations between $S(x,\,Q)$ and $\psi$; we give the statement as stated in \cite{harman1998metric}:

\begin{thm*}[Schmidt, 1960]
    Let $\psi:\mathbb{N} \to [0,1/2)$. Suppose $\psi$ is non-increasing and \eqref{sum} diverges. Then for any $\varepsilon>0$ and almost every $x\in[0,1)$, as $Q\to\infty$ we have that
    \begin{align*}
    S(x,Q)=2\Psi(Q)+O_{\psi,\varepsilon,x}\left(\Psi^{1/2}(Q)\log^{2+\varepsilon}{\Psi(Q)}\right),
    \end{align*}
where for any $Q\in\mathbb{N}$, $\Psi(Q)$ is defined to be   
\begin{align*}
    \Psi(Q)\coloneqq\sum_{q=1}^Q\psi(q).
 \end{align*}
\end{thm*}

Roughly speaking, Schmidt's Theorem tells us that for almost all $x\in[0,1)$, the number of solution pairs $(p,q)$ to \eqref{psi approximable} such that $1\leq q\leq Q$ is approximately equal to $2\Psi(Q)$ when $Q$ is sufficiently large. 

The error term above involves an implicit constant which depends on the point $x$. However, establishing an explicit constant for this result, and also some of its preliminary results, has the potential of establishing more research in some other fields; one example is that these results may be useful is signal processing (see \cite{beresnevich2020number}).

In this paper we prove the following quantitative version of Schmidt's Theorem.

\begin{thm}[Effective Schmidt's Theorem]\label{EST} 
Let $\psi:\mathbb{N}\to[0,1/2)$ be a non-increasing function. Suppose \eqref{sum} diverges. Then for any $\varepsilon>0$ and $\delta>0$, there exists a measurable $E_{\varepsilon,\delta}\subset[0,1)$ such that $\mu(E_{\varepsilon,\delta})<\delta$ and for any $x\in[0,1)\setminus E_{\varepsilon,\delta}$ and $Q\in\mathbb{N}$, 
\begin{align*}
    \left|S(x,Q)-2\Psi(Q)\right|\leq 
    \max{\left(
    N_{\varepsilon,\delta},\,
    K_{\varepsilon}\left(\Psi^{1/2}(Q)\log^{2+\varepsilon}(\Psi(Q)+1)\right)\right)},
\end{align*}
where
\begin{align*}
    N_{\varepsilon,\delta}&=2\max{\left\{n\in\mathbb{N}:\left\lceil\left(\frac{2K_{\varepsilon}}{\varepsilon\delta}\right)^{1/\varepsilon}\right\rceil+1>44\Psi(n)\log{(3\Psi^2(n)+3)}\log{(2\log{(3\Psi^2(n)+3)})}\right\}}, \\
    K_{\varepsilon}&=\frac{28+858(\varepsilon+1)7^\varepsilon}{{\psi}^{1/2}(1)\log^{2+\varepsilon}{(\psi(1)+1)}}.
\end{align*}
\end{thm}
Roughly speaking, this tells us that other than on a set of Lebesgue measure less than $\delta$ we can give an effective double bound for the size of $S(x,\,Q)$. We further note that if we let $\delta \to 0$ and $Q \to \infty$ we regain Schmidt's Theorem. It can be viewed on an effective version of the quantitative Borel-Cantelli lemma which is a standard tool for establishing quantitative statements such as Schmidt's Theorem.

The key tool to proving this result is a probabilistic result that we will give in Section 2. The structure of the rest of the paper is as follows. In Section 2 we will give the probabilistic results we use and will immediately prove Theorem \ref{EST}. In Section 3 we will demonstrate the wide applicability of the probabilistic results by giving a range of applications as follows:
\begin{itemize}
    \item Quantative Koukoulopoulos-Maynard Theorem (see \cite{aistleitner2022metric})
    \item Applications to $M_{0}$-sets (see \cite{Pollington2022})
    \item Normal numbers
    \item Strong law of large numbers
\end{itemize}
Finally, in section 4 we will give the proofs of the probabilistic results before giveing supplementary proofs and a general form of the asymptotic results that can be widely applied in the appendix.

The reader who is solely interested in the probabilistic results need only refer to the start of Section 2 and to Section 4; however we hope the applications to number theoretic problems proves interesting to this reader too.

\section{Probabilistic Results and Proof of an Effective Schmidt's Theorem}
\subsection{Probabilistic Results}
The main tool in proving the asymptotic results above are results of the kind of Lemma 1.4 and 1.5from \cite{harman1998metric}. Here we state effective versions of these lemmas which will allow us to give the proof of Theorem \ref{EST}, as well as the further applications in Section 3; we leave the proof of these statements to section 4.

\begin{thm}[Effective Version of Lemma 1.4 from \cite{harman1998metric}]\label{thm2} 
Let $(X,\,\Omega,\,\mu)$ be a measure space, and suppose that $0<\mu(X)<+\infty$. Let $f_k(x)$, $k \in \mathbb{N},$ be a sequence of non-negative $\mu$-measurable functions, where for all $k \in \mathbb{N},\, x \in X,$ we have that  $f_{k}(x) <C$ for some constant $C\in \mathbb{R}$. Let $\{f_k\in\mathbb{R}\}_{k\in\mathbb{N}}$ and $\{\varphi_k\in\mathbb{R}\}_{k\in\mathbb{N}}$ be sequences of real numbers such that for any $k\in\mathbb{N}$, 
\begin{equation}\label{first conditions on fk, phik}
    0\leq f_k\leq \varphi_k\leq 1.
\end{equation} 
For any $N\in\mathbb{N}$, define
\[\varPhi(N)=\sum_{k=1}^N\varphi_k,\] 
and suppose that $\lim_{N\to\infty}\varPhi(N)=+\infty$. Further, assume that there exists $K>0$ such that for any $N\in\mathbb{N}$ we have that
\begin{align}\label{1.2.7}
\int_X\left(\sum_{k=1}^N(f_k(x)-f_k)\right)^2\,\mathrm{d}\mu(x)\leq K\varPhi(N).
\end{align}
Then for any $\varepsilon>0$ and $\delta>0$, there exists $E_{\varepsilon,\delta}\subset X$ and $K_{\varepsilon,\delta}>0$ such that $\mu(E_{\varepsilon,\delta})<\delta$ and for any $x\in X\setminus E_{\varepsilon,\delta}$ and $N\in\mathbb{N}$,
\begin{align*}
\left|\sum_{k=1}^Nf_k(x)-\sum_{k=1}^Nf_k\right|\leq K_{\varepsilon,\delta}\left(\varPhi^{2/3}(N)\log^{1/3+\varepsilon}{(\varPhi(N)+2)}\right),
\end{align*}
where 
\begin{align*} 
K_{\varepsilon,\delta}&=\max \left\{\frac{CN_{\varepsilon,\delta}}{\max{\left(\varPhi_0^{2/3}\log^{1/3+\varepsilon}(\varPhi_0+2),1\right)}},\, 
\frac{4}{\log^{2\varepsilon/3}{(\varPhi_0+2)}}\left(\frac{\log{4}}{\log{3}}\right)^{1+\varepsilon}\left(4+\frac{1+\varepsilon}{\log{3}}+\frac{1}{4\log^{1+\varepsilon}{4}}\right)\right\},\\
N_{\varepsilon,\delta}&=\min{\left\{n\in\mathbb{N}:\varPhi(n)>j^3_{\varepsilon,\delta}\log^{1+\varepsilon}(j_{\varepsilon,\delta}+2)\right\}},\\ j_{\varepsilon,\delta}&=1+\left\lceil\exp{\left(\frac{1+\log^{-1-\varepsilon}{3}}{\varepsilon\delta}K\right)^{1/\varepsilon}}\right\rceil,
\end{align*}
where $\varPhi_{0}=\min{\{\varPhi(n)\in\mathbb{R}^+:n\in\mathbb{N}\}}$ is the minimal non-zero value $\varPhi$ attains.
\end{thm}

Similarly, we are able to come up with an effective version of Lemma 1.5 from \cite{harman1998metric}, as follows.

\begin{thm}[Effective Version of Lemma 1.5 from \cite{harman1998metric}]\label{thm3} 
Let $(X,\Omega,\mu)$ be a measure space, and suppose that $0<\mu(X)<+\infty$. Let $f_k(x),\, k\in \mathbb{N}$, be a sequence of non-negative $\mu$-measurable functions, where for all $k \in \mathbb{N},\, x \in X,$ we have that $f_{k}(x) <C$ for some constant $C\in \mathbb{R}$. Let $\{f_k\in\mathbb{R}\}_{k\in\mathbb{N}}$ and $\{\varphi_k\in\mathbb{R}\}_{k\in\mathbb{N}}$ be sequences of real numbers such that, for all $k\in\mathbb{N}$, 
\begin{equation*}\label{second fk phik condition}
0\leq f_k\leq \varphi_k.    
\end{equation*}
For any $N\in\mathbb{N}$, let
\[\varPhi(N)=\sum_{k=1}^N\varphi_k.\]
and suppose that $\lim_{N\to\infty}\varPhi(N)=+\infty$. Further, assume that there exists $K>0$ such that for any $m,n\in\mathbb{N}$, if $m<n$ then
\begin{align}\label{1.2.7.a}
\int_X\left(\sum_{k=m+1}^n(f_k(x)-f_k)\right)^2\,\textrm{d}\mu(x)\leq K\left(\varPhi(n)-\varPhi(m)\right).
\end{align}
Then for any $\varepsilon>0$ and $\delta>0$, there exists $E_{\varepsilon,\delta}\subset X$ and $K_{\varepsilon,\delta}>0$ such that $\mu(E_{\varepsilon,\delta})<\delta$ and for any $x\in X\setminus E_{\varepsilon,\delta}$ and $N\in\mathbb{N}$,
\begin{align*}
\left|\sum_{k=1}^Nf_k(x)-\sum_{k=1}^Nf_k\right|\leq K_{\varepsilon,\delta}\left(\varPhi^{1/2}(N)\log^{3/2+\varepsilon}{\varPhi(N)}+\max_{1\leq k\leq N}{f_k}\right),
\end{align*}
where
\begin{align*}
K_{\varepsilon,\delta}&=\max \left\{\frac{CN_{\varepsilon,\delta}}{\max{\left(\varPhi_0^{1/2}\log^{3/2+\varepsilon}(\varPhi_0+2)+f_1,1\right)}},\,\frac{2}{\log^{3/2+\varepsilon/2}{2}}\left(1+\frac{1}{\sqrt{2}\log^{3/2+\varepsilon}{4}}\right)
\left(\frac{\log{4}}{\log{3}}\right)^{3/2+\varepsilon}\right\},\\
N_{\varepsilon,\delta}&=\max{\{n\in\mathbb{N}:\varPhi(n)<r_{\varepsilon,\delta}\}}, \\
r_{\varepsilon,\delta}&=\left\lceil\left(\frac{2K}{\varepsilon\delta}\right)^{1/\varepsilon}\right\rceil+1,
\end{align*}
where $\varPhi_{0}=\min{\{\varPhi(n)\in\mathbb{R}^+:n\in\mathbb{N}\}}$ is the minimal non-zero value $\varPhi$ attains.
\end{thm}

We will give the proofs of these theorems in Section 4.

\subsection{Proof of Theorem \ref{EST} modulo Theorem \ref{thm3}}
We now prove Theorem \ref{EST}. The proof given in this section is based on Section 4.1from \cite{harman1998metric}; the proof given there is an improvement based on Schmidt's original method.

Let $\psi:\mathbb{N}\to(0,1/2]$ be a non-increasing function. Suppose \eqref{sum} diverges. We begin by making the Lemmas 4.1, 4.2 and 4.3 from \cite{harman1998metric} effective.

\begin{lemma}[Effective Lemma 4.1 from \cite{harman1998metric}]\label{Lemma 4.1}
For any $M,N,k\in\mathbb{N}$, if $M<N$ then 
    \begin{align}\label{4.2.2}
        0\leq N-M+1-\sum_{n=M}^N\frac{\varphi(k,n)}{n}\leq\frac{N-M}{k}+\log{N}, 
    \end{align}
where 
    \begin{align*}
        \varphi(k,n)\coloneqq\sum_{\substack{1\leq m\leq n,\\\gcd{(m,n)}\leq k}}1. 
    \end{align*}
\end{lemma}

\begin{proof}
The lower bound follows from the fact that for any $k,n\in\mathbb{N}$, $\varphi(k,n) \leq n$. We now focus on the upper bound.

Notice that for any $k,n\in\mathbb{N}$, \begin{align*}
    \varphi(k,n)\geq n-\sum_{\substack{d\,|\,n\\d>k}}\sum_{\substack{m=1\\m\equiv0\,\operatorname{mod}d}}^n1.
\end{align*}

Thus, for any $M,N\in\mathbb{N}$, if $M<N$ then \begin{align*}
    \sum_{n=M}^N\frac{\varphi(k,n)}{n}
    &\geq\sum_{n=M}^N\left(1-\frac{1}{n}\sum_{\substack{d\,|\,n\\d>k}}\frac{n}{d}\right) \\
    &=N-M+1-\sum_{d=k+1}^N\left(\frac{1}{d}\sum_{\substack{n=M\\n\equiv0\operatorname{mod}d}}^N1\right) \\
    &\geq N-M+1-\sum_{d=k+1}^N\left(\frac{1}{d}\left(\frac{N-M}{d}+1\right)\right) \\
    &=N-M+1-\sum_{d=k+1}^N\frac{N-M}{d^2}-\sum_{d=k+1}^N\frac{1}{d}\\
    &\geq N-M+1-\frac{N-M}{k}-\log{N}
\end{align*}

\end{proof}

\begin{lemma}[Effective Lemma 4.2 from \cite{harman1998metric}]\label{Lemma 4.2}
Given an eventually non-increasing function $\psi :\mathbb{N} \to \mathbb{R}$, for any $M,N,k\in\mathbb{N}$, if $M<N$ then 
    \begin{align*}
        \left(1-\frac{1}{k}\right)\sum_{n=M}^{N} \psi(n)-\psi(M)\log{M}-\sum_{n=M}^{N}\frac{\psi(n)}{n}
        \leq \sum_{n=M}^N \frac{\psi(n)\varphi(k,n)}{n}
        \leq \sum_{n=M}^N\psi(n).
    \end{align*}
\end{lemma}

\begin{proof}
Pick any $M,N,k\in\mathbb{N}$ with $M<N$. By partial summation, 
\begin{align}\label{partial summation}
    \sum_{n=M}^N\frac{\psi(n)\varphi(k,n)}{n}
    =\sum_{n=M}^{N-1}\left((\psi(n)-\psi(n+1))\sum_{m=M}^n\frac{\varphi(k,m)}{m}\right)+\psi(N)\sum_{m=M}^N\frac{\varphi(k,m)}{m}
\end{align}

We recall that $\frac{\varphi(k,\,n)}{n} \leq 1$, and applying this to the above and summing over $n$ we find that
\begin{align*}
    \sum_{n=M}^{N-1}\left((\psi(n)-\psi(n+1))\sum_{m=M}^n\frac{\varphi(k,m)}{m}\right)+\psi(N)\sum_{m=M}^N\frac{\varphi(k,m)}{m} \leq \sum_{n=M}^{N}\psi(N),
\end{align*}
which establishes the upper bound of the lemma.

To establish the lower bound, we consider the upper bound of \eqref{4.2.2} akin to the above, and we applying this to \eqref{partial summation}, we obtain that
\begin{align*}
    \sum_{n=M}^N\frac{\psi(n)\varphi(k,n)}{n}
    &=\sum_{n=M}^{N-1}\left((\psi(n)-\psi(n+1))\sum_{m=M}^n\frac{\varphi(k,m)}{m}\right)+\psi(N)\sum_{m=M}^N\frac{\varphi(k,m)}{m} \\
    &\geq \sum_{n=M}^{N-1}\left(\left(\psi(n)-\psi(n+1)\right)(n-M+1-\frac{n-M}{k}\log n)\right) \\
    &=\sum_{n=M}^{M}\psi(n)-\frac{1}{k}\sum_{m=M+1}^{N}\psi(m)-\psi(M) \log M +\sum_{l=M+1}^{N}\psi(l)\left(\log l-\log(l+1)\right) \\
    &\geq \left(1-\frac{1}{k}\right)\sum_{n=M}^{N} \psi(n)-\psi(M)\log{M}-\sum_{n=M}^{N}\left(\frac{\psi(n)}{n}\right),
\end{align*}

where the last line follows as $\log l -\log (l+1)>-\frac{1}{l}$ for all positive $l \in \mathbb{N}$
\end{proof}

We now give some definitions before stating and proving Lemma \ref{Lemma 4.3}. For any $N\in\mathbb{N}$, define
\begin{align}
    \Psi(N)    &\coloneqq\sum_{n=1}^N\psi(n), \nonumber\\
    \Gamma(N)  &\coloneqq\Psi^2(N)+1        \geq1, \nonumber\\
    L(N)       &\coloneqq\log{(3\Gamma(N))} \geq\log{3}>1,\nonumber\\
    L_2(N)     &\coloneqq\log{(2L(N))}      \geq\log{(2\log{3})}>1/2, \nonumber\\
    \varPhi(N) &\coloneqq\sum_{\substack{1\leq m\leq N,\\\gcd{(m,N)}\leq \Gamma(N)}}1\leq N. \label{PHI DeFN}
\end{align}

\begin{lemma}[Effective Lemma 4.3 from \cite{harman1998metric}]\label{Lemma 4.3}
    For any $N\in\mathbb{N}$, 
    \begin{align*}
        0\leq\sum_{n=1}^N\psi(n)\left(1-\frac{\varPhi(n)}{n}\right)\leq 40.6 L(N)L_2(N). 
    \end{align*}
\end{lemma}
\begin{proof} The first inequality follows directly from \eqref{PHI DeFN}, as for any $N\in\mathbb{N}$, 
\begin{align*}
    \sum_{n=1}^N\frac{\psi(n)\varPhi(n)}{n}\leq\sum_{n=1}^N\psi(n).
\end{align*}
The rest of the proof deals with the second inequality. Pick any $N\in\mathbb{N}$. Define $M_0\coloneqq0$ and let
\begin{align*}
    M_j &\coloneqq 2^{2^j}, \\
    h   &\coloneqq \min\{j\in\mathbb{N}:\Psi(N)<M_j\}\geq1,\\
    V_r &\coloneqq \min\{V\in\mathbb{N}\cup\{0\}:\Psi(V)\geq M_r\}. 
\end{align*}

Suppose initially that $h\geq 2$. By the definition of $h$ we have that 
\[2^{h-1}\log{2}=\log{M_{h-1}}\leq\log{\Psi(N)}<0.5L(N).\]
It follows that $2^{h+1}\log{2}<2L(N)$, so upon taking logs, we obtain that
\begin{align*}
    (h+1)\log{2}+\log{\log{2}}&<L_2(N) \\
    h&<\frac{L_2(N)-\log{\log{2}}}{\log{2}}-1 <\frac{L_2(N)}{\log{2}}
\end{align*}
and, from the lower bound that $L_{2}(N)\leq\log{(2\log{3})}$ that,
\begin{align*}
    3+h&<\left(\frac{3}{\log{(2\log{3})}}+\frac{1}{\log{2}}\right)L_2(N). 
\end{align*}
In the case that $h=1$, these inequalities are trivially true.

We consider the terms in the right hand side of the above, noting that if $M_r\leq\Psi(n)$ then ${M_r}^2+1\leq\Gamma(n)$ and $\varphi({M_r}^2+1,n)\leq\varPhi(n)$. Applying these and Lemma \ref{Lemma 4.2}, it follows that
\begin{align*}
    \sum_{n=1}^N\frac{\psi(n)\varPhi{(n)}}{n}
    &\geq\mathop{\sum_{0\leq r\leq h-1}\sum_{1\leq n\leq N}}_{M_r\leq\Psi(n)<M_{r+1}}\frac{\psi(n)\varphi({M_r}^2+1,n)}{n} \\
    &\geq\sum_{n=1}^N\psi(n)-\left(\sum_{n=1}^N\frac{\psi(n)}{n}+\sum_{r=0}^{h-1}\left(\frac{1}{{M_r}^2+1}\sum_{\substack{1\leq n\leq N\\M_r\leq\Psi(n)<M_{r+1}}}\psi(n)+\psi(V_r)\log{V_r}\right)\right).
\end{align*}

Considering the sums in the final line of the above, we note that 
\begin{align*}
    \sum_{r=0}^{h-1}\left(\frac{1}{{M_r}^2+1}\sum_{\substack{1\leq n\leq N, \\M_r\leq\Psi(n)<M_{r+1}}}\psi(n)\right)<\sum_{r=0}^{h-1}\frac{M_{r+1}}{{M_r}^2+1}<3+h<\left(\frac{3}{\log{(2\log{3})}}+\frac{1}{\log{2}}\right)L_2(N), 
\end{align*}
and the assumption that $\psi$ is non-increasing gives us that
\begin{align*}
    \sum_{r=0}^{h-1}\psi(V_r)\log{V_r}
    &\leq \sum_{r=0}^{h-1}\left(\psi(V_r)\sum_{n=1}^{V_r}\frac{1}{n}\right) 
    \leq \sum_{n=1}^N\left(\frac{1}{n}\sum_{r=1,\,V_r>n}^{h-1}\psi(V_r)\right) \\
    &\leq \sum_{n=1}^N\frac{\psi(n)}{n}h 
    \leq \frac{L_2(N)}{\log{2}}\sum_{n=1}^N\frac{\psi(n)}{n}.
\end{align*}

Applying H\"{o}lder's inequality, as $1/L(N)<1$ we obtain that
\begin{align*}
    \sum_{n=1}^N\frac{\psi(n)}{n}
    &\leq\left(\sum_{n=1}^N\psi(n)^{1+L(N)}\right)^{1/(1+L(N))}\left(\sum_{n=1}^N\frac{1}{n^{1+1/L(N)}}\right)^{L(N)/(1+L(N))} 
\end{align*}
We recall that for any $\varepsilon<1,$
\[\sum_{n=1}^{N}n^{-1-\varepsilon}<2\int_{1/2}^{\infty}\frac{1}{x^{1+\varepsilon}}\textrm{d}x<\frac{8}{\varepsilon}.\]
Substituting this into the above, with $\varepsilon=1/L(N)$, we obtain that 
\begin{align*}
    \sum_{n=1}^N\frac{\psi(n)}{n}&\leq8\Psi(N)^{1/L(N)}L(N)\leq8\sqrt{e}L(N).
\end{align*}

The above follows as, letting $x=\Psi(N)$, we note that $\Psi(N)^{1/L(N)}$ is of the form $ x^{1/\log{(3x^2+3)}}$. By standard methods in calculus, we note that for any $x\geq0$, \begin{align*}
    x^{1/\log{(3x^2+3)}}\leq\sqrt{e}. 
\end{align*}

Combining the results above, we get that
\begin{align*}
    \sum_{n=1}^N\psi(n)\left(1-\frac{\varPhi(n)}{n}\right)
    &\leq 8\sqrt{e}L(N)+\left(\frac{3}{\log{(2\log{3})}}+\frac{1}{\log{2}}\right)L_2(N)+\frac{1}{\log{2}}L_2(N)8\sqrt{e}L(N) \\
    &\leq \left(\frac{8\sqrt{e}}{\log{(2\log{3})}}+\frac{3}{\log{3}\log{(2\log{3})}}+\frac{1}{\log{2}\log{3}}+\frac{8\sqrt{e}}{\log{2}}\right)L(N)L_2(N). 
\end{align*} 
The result then follows by computing the value 
\begin{align*}
    8\sqrt{e}\left(\frac{1}{\log\left(2\log3\right)}+\frac{1}{\log{2}}\right)+\frac{3}{\log{3}\log{(2\log3)}}+\frac{1}{\log{2}\log{3}}=40.56633883...<40.6. 
\end{align*}
\end{proof}

We are now in a position to prove Theorem \ref{EST}.

\begin{proof}[Proof of Theorem \ref{EST}] 
Suppose Theorem \ref{thm3} is true.

Denote by $\mu$ the Lebesgue measure on $[0,1)$ and let $S^*(\alpha,u,v)$ be the number of solutions $n\in\mathbb{N}$ to 
\begin{align*}
    |\alpha n-m|<\psi(n),\quad \gcd{(m,n)}\leq \Gamma(n),\quad u<n\leq v.
\end{align*}

We now compute
\[\int_{[0,1)}S^{*}(\alpha,0,N)\,\textrm{d}\alpha\]
as follows. For any $u,v\in\mathbb{N}$, if $u<v$ then
\begin{align*}
    \int_{[0,1)}S^*(\alpha,u,v)\,\textrm{d}\alpha=\sum_{n=u+1}^v\int_{[0,1)}S^*(\alpha,n-1,n)\,\textrm{d}\alpha=2\sum_{n=u+1}^v\psi(n)\frac{\varPhi(n)}{n}.
\end{align*}

We recall that, by the definition of $S$, we have  that,
\[\int_{[0,1)} S(\alpha,\,n) \,\textrm{d}\alpha = \psi(n).\]
Notice that for any $N\in\mathbb{N}$, $S(\alpha,N)\geq S^*(\alpha,0,N)$. By Lemma \ref{Lemma 4.3}, we get that for any $v\in\mathbb{N}$,
\begin{align*}
    \int_{[0,1)}\left(S(\alpha,v)-S^*(\alpha,0,v)\right)\,\textrm{d}\alpha=2\sum_{n=1}^{v}\psi(n)\left(1-\frac{\varPhi(n)}{n}\right)\leq 81.2L(v)L_2(v). 
\end{align*}

For any $n\in\mathbb{N}$, define $G:\mathbb{N}\to\mathbb{R}^+$ by $G(n)=L(n)/L_2(n)$. Thus, for any $N\in\mathbb{N}$,
\begin{align*}
    \mu\left(\left\{\alpha\in[0,1):S(\alpha,N)-S^*(\alpha,0,N)>L^2(N)\right\}\right)\leq \frac{81.2}{G(N)}. 
\end{align*}

For any $j\in\mathbb{N}\cup\{0\}$, let $v_j=\min{\{n\in\mathbb{N}:L(n)\geq2^j\}}$. Then 
\begin{align*}
    \sum_{j=0}^\infty \frac{1}{G(v_j)}<+\infty. 
\end{align*}

Hence, by the Borel-Cantelli Lemma,  for almost every $\alpha \in[0,1)$ and $j>j(\alpha)$,
\begin{align*}
    0<S(\alpha,v_j)-S^*(\alpha,0,v_j)<L^2(v_j). 
\end{align*}

Note that 
\[2^{j-1}\leq L(v_{j-1}) \leq L(N)\leq 2^{j} \leq L(v_{j}) \leq 2^{j+1}.\]
It follows that $4L(N)\geq 2^{j+1} \geq L(v_{j})$, so $16L^{2}(N) \geq L^{2}(v_{j})$.
Thus, for almost every $\alpha\in[0,1)$, 
\begin{align} \label{|S-S^*|}
    0<S(\alpha,N)-S^*(\alpha,0,N)\leq 16L^2(N). 
\end{align}


We now study $S^*(\alpha,0,N)$. We are going to apply Theorem \ref{thm3} on $S^*(\alpha,0,N)$. Pick any $u,v\in\mathbb{N}$ with $u<v$. Notice that for any $n\in\mathbb{N}$, \[1-\frac{\varPhi(n)}{n}\leq1\leq 2L(n)L_2(n).\] 

Let
\begin{align*}
    d^*(n)   &\coloneqq\sum_{d\,|\,n,\,1\leq d\leq\Gamma(n)}1 \\
    \Psi(u,v)&\coloneqq\sum_{u<n\leq v}\psi(n).
\end{align*}
Then
\begin{align*}
    \Psi(u,v)\sum_{n=u+1}^v\psi(n)\left(1-\frac{\varPhi(n)}{n}\right)
    &=\sum_{n=u+1}^n\psi(n)\left(1-\frac{\varPhi(n)}{n}\right)+\sum_{r=u+1}^v\psi(r)\left(1-\frac{\varPhi(r)}{r}\right)\sum_{m=u+1}^v\psi(m)\\
    &\leq\sum_{n=u+1}^v2\psi(n)L(n)L_2(n)+\sum_{r=u+1}^v\psi(r)\left(1-\frac{\varPhi(r)}{r}\right)\Psi(r). 
\end{align*}

Hence, the assumptions of Theorem 3 are satisfied, as by Lemma 4.4 and inequality (4.2.10) in \cite{harman1998metric}, 
\begin{align*}
    \int_{[0,1)}\left(S^*(\alpha,u,v)-2\Psi(u,v)\right)^2\textrm{d}\alpha
    &\leq4\sum_{u<n\leq v}d^*(n)\psi(n)+4\Psi(u,v)\sum_{u<n\leq v}\psi(n)\left(1-\frac{\varPhi(n)}{n}\right) \\
    &\leq4\sum_{u<n\leq v}\left(d^*(n)+2L(n)L_2(n)+\left(1-\frac{\varPhi(n)}{n}\right)\Psi(n)\right)\psi(n), 
\end{align*}

By Theorem \ref{thm3}, for any $\varepsilon>0$ and $\delta>0$ there exists a measurable $E_{\varepsilon,\delta}\subset[0,1)$ such that $\mu(E_{\varepsilon,\delta})<\delta$ and for any $x\in[0,1)\setminus E_{\varepsilon,\delta}$ and $N\in\mathbb{N}$, \begin{align*}
    \left|S^*(\alpha,0,N)-2\Psi(N)\right|\leq K_{\varepsilon,\delta,0}\left({\Psi_1}^{1/2}(N)\log^{3/2+\varepsilon}{\Psi_1(N)}+\frac{1}{2}\right)
, \end{align*}
where by Lemma \ref{Lemma 4.3}, with $\psi_0\coloneqq\psi(1)>0$, we have
\begin{align*}
    K_{\varepsilon,\delta,0}&=\frac{4}{\log^{2\varepsilon/3}{(\psi_0+2)}}\left(\frac{\log{4}}{\log{3}}\right)^{1+\varepsilon}\left(4+\frac{1+\varepsilon}{\log{3}}+\frac{1}{4\log^{1+\varepsilon}{4}}\right) \\
    &\leq26(\varepsilon+1)\left(\frac{\log{4}}{\log{3}\log^{2/3}{2}}\right)^\varepsilon 
    \leq26(\varepsilon+1)1.75^\varepsilon,\\
    \Psi_1(N)
    &\coloneqq\sum_{n=1}^N\left(d^*(n)+2L(n)L_2(n)+\left(1-\frac{\varPhi(n)}{n}\right)\Psi(n)\right)\psi(n)\\
    &\leq\sum_{n=1}^Nd^*(n)\psi(n)+42.6\Psi(N)L(N)L_2(N).
\end{align*}

Notice that for any $N\in\mathbb{N}$, 
\begin{align*}
    \sum_{n=1}^Nd^*(n)\psi(n)
    &=\sum_{n=1}^N\psi(n)\sum_{d\,|\,n,\,d\leq\Gamma(n)}1 = \sum_{d=1}^{\Gamma(N)}\sum_{n=1}^{N/d}\psi(kd)\\
    &\leq\sum_{d\leq\Gamma(N)}\frac{\Psi(N)}{d} \\
    &\leq\Psi(N)\log{(3\Gamma(N))}=\Psi(N)L(N)\\
    &\leq\frac{1}{\log{(2\log{3})}}\Psi(N)L(N)L_2(N). 
\end{align*}

Thus we find that
\begin{align*}
    \Psi_1(N)
    \leq \left(42.6+\frac{1}{\log{(2\log{3})}}\right)\Psi(N)L(N)L_2(N)
    <44\Psi(N)L(N)L_2(N),
\end{align*}
and taking 
\begin{align*}
    \psi_0'=\Psi_1(1)
    \leq44\psi(1)\log\left(3(\psi^2(1)+1)\right)\log{(2\log\left({3(\psi^2(1)+1)}\right))}
    \leq28.268, 
\end{align*}
we find that
\begin{align*}
    {\Psi_1}^{1/2}(N)\log^{3/2+\varepsilon}{\Psi_1(N)}+\frac{1}{2}
    &\leq\frac{{\psi_0'}^{1/2}\log^{3/2+\varepsilon}{\psi_0'}+1/2}{{\psi_0}^{1/2}\log^{2+\varepsilon}{(\psi_0+1)}}{\Psi}^{1/2}(N)\log^{2+\varepsilon}{(\Psi(N)+1)} \\
    &\leq\frac{33(4)^\varepsilon}{{\psi_0}^{1/2}\log^{2+\varepsilon}{(\psi_0+1)}}{\Psi}^{1/2}(N)\log^{2+\varepsilon}{(\Psi(N)+1)}.
\end{align*}

Hence, by the triangle inequality and \eqref{|S-S^*|}, 
\begin{align*}
    |S(\alpha,N)-2\Psi(N)|
    &\leq|S(\alpha,N)-S^*(\alpha,0,N)|+|S^*(\alpha,0,N)-2\Psi(N)| \\
    &\leq16L^2(N)+\frac{33(4)^\varepsilon K_{\varepsilon,\delta,0}}{{\psi_0}^{1/2}\log^{2+\varepsilon}{(\psi_0+1)}}{\Psi}^{1/2}(N)\log^{2+\varepsilon}{(\Psi(N)+1)}\\
    &\leq\left(\frac{16(\log{(3{\psi_0}^2+3)})^2}{{\psi_0}^{1/2}\log^{2+\varepsilon}{(\psi_0+1)}}+\frac{33(4)^\varepsilon K_{\varepsilon,\delta,0}}{{\psi_0}^{1/2}\log^{2+\varepsilon}{(\psi_0+1)}}\right){\Psi}^{1/2}(N)\log^{2+\varepsilon}{(\Psi(N)+1)} \\
    &\leq\frac{28+33(4)^\varepsilon K_{\varepsilon,\delta,0}}{{\psi_0}^{1/2}\log^{2+\varepsilon}{(\psi_0+1)}}{\Psi}^{1/2}(N)\log^{2+\varepsilon}{(\Psi(N)+1)} \\
    &\leq K_{\varepsilon,\delta}{\Psi}^{1/2}(N)\log^{2+\varepsilon}{(\Psi(N)+1)}.
\end{align*}
This proves the theorem by the established upper bound for $K_{\varepsilon,\delta,0}$. 
\end{proof}

\section{Further Applications}
In this section we give a range of further applications of Theorems \ref{thm2} and \ref{thm3}. We begin by returning to our discussion of the Koukoulopoulos-Maynard Theorem (Duffin-Schaeffer Conjecture). 

\subsection{Quantitative Koukoulopoulos-Maynard Theorem}

Expanding on the work of Koukoulopoulos and Maynard, an asymptotic relation for the quantitative Duffin-Schaeffer type problem is given by Aistleitner, Borda and Hauke in \cite{aistleitner2022metric}; that is, the paper gives an asymptotic description of the number of solutions satisfying \eqref{psi approximable} with the extra condition that $\gcd{(p,q)}=1$. They proved the following theorem.

\begin{thm*}[Aistleitner, Borda, Hauke, 2022] 
    Let $\psi:\mathbb{N} \to [0,1/2]$. Suppose \eqref{sum'} diverges and let $C>0$ be arbitrary. Then for almost every $x\in[0,1)$, we have that
\begin{align}\label{Aistleitner}
    S'(x,Q)=\Psi'(Q)\left(1+O_{\psi,C,x}\left(\frac{1}{\left(\log(\Psi'(Q))\right)^{C}}\right)\right)
, \end{align}
as $Q\to\infty$, where   
\begin{align} \label{definition of S'}
    S'(x,Q)
    &\coloneqq\#\left\{q\in\mathbb{N}\cap[1,Q]:\left|x-\frac{p}{q}\right|<\frac{\psi(q)}{q}\textrm{, for some }p\in\mathbb{N}\textrm{ such that }\gcd{(p,q)}=1.\right\}, 
\end{align}
    and $\Psi'(Q)$ is defined to be
\begin{align*}
    \label{definition of Psi'}
      \Psi'(Q)
    &\coloneqq2\sum_{q=1}^Q\frac{\psi(q)\varphi(q)}{q}.   
\end{align*}
\end{thm*}

Using the tools from the previous section, we are able to give an effective version of this theorem.
\begin{thm}[Effective Aistleitner-Borda-Hauke Theorem]\label{effective solution counting}
Let $\psi:\mathbb{N}\to[0,1/2]$ be a function and let $C>0$ be arbitrary. Suppose \eqref{sum'} diverges. Then there exists a measurable $E_{C,\,\delta}\subset[0,1)$ such that $\mu\left(E_{C,\,\delta}\right)< \delta$, and for any $x \in [0,\,1) \setminus E_{C,\,\delta}$ we have that
    \[\left|S'(x,Q)- \Psi'(Q)\right|\leq \max\left\{\frac{k_{C,\,\delta}}{2},\,\frac{2e\Psi'(Q)+1 }{\left(\log \Psi'\left(Q\right)\right)^{C}}+\frac{1}{2}\right\},\]
up to a constant depending only on $C$ found following the proof of Theorem 2 from \cite{aistleitner2022metric}, where $S'(x,Q)$ and $\Psi'(Q)$ are given above, and $k_{C,\,\delta}$ is given in \eqref{k_EABHT}.
\end{thm}

The proof involves an application of a result akin to Harman's Lemma 1.4 \cite{harman1998metric}, which we give explicitly here. We note the version below is given in a more general setting than in \cite{aistleitner2022metric}.
 
 \begin{lemma} \label{Aist general thm}
 Suppose that $\left(X,\,\Omega,\,\mu\right)$ is a measure space such that $0<\mu\left(X\right)< \infty$ and let $K$ be a positive real constant. Let $f_{q}: X \to  \left[0,\,K \right]$ be a sequence of $\mu$-measurable functions and let $f_{q},\, \phi_{q}$ be sequences of real numbers such that
 \[0 \leq f_{q} \leq \phi_{q} \leq K.\]
 Write $\Psi(Q)= \sum_{q=1}^{Q} \phi_{q}$, and suppose that $\Psi(Q) \rightarrow \infty$ and $Q \rightarrow \infty$.  Let $C>0$ be arbitrary. Suppose that for every $Q\in\mathbb{N}$,
 \begin{equation}\label{Aist int}
    \int_{X}\left(\sum_{1\leq q \leq Q}\left(f_{q}(x)-f_{q}\right)\right)^{2} \textrm{d}\mu = O_C\left(\frac{\Psi(Q)^{2}}{\left(\log \Psi(Q)\right)^{C}}\right). 
 \end{equation}

 Then for almost all $x \in X$, as $Q \to \infty$,
 \[\sum_{1\leq q \leq Q}f_{q}(x)=\sum_{1 \leq q \leq Q}f_{q} +O\left(\frac{\Psi(Q)}{\left(\log \Psi(Q)\right)^{C}}\right).\]
 \end{lemma}
 
This is proved in the necessary case in Section 2 from \cite{aistleitner2022metric}, following the method of proof of Lemma 1.4 from \cite{harman1998metric}. In Section 4, we prove Lemma \ref{Aist general thm}, and we then prove Theorem \ref{effective v3} (its effective version); these are left to the next section as they are variants on the probabilistic results given above. 

\begin{thm}\label{effective v3} 
Suppose that $\left(X,\,\Omega,\,\mu\right)$ is a measure space such that $0<\mu\left(X\right)< \infty$, and let $K$ be a positive real constant. Let $f_{q}: X \to  \left[0,\,K \right]$ be a sequence of $\mu$-measurable functions and let $f_{q},\, \phi_{q}$ be sequences of real numbers such that
\[0 \leq f_{q} \leq \phi_{q} \leq K.\]
Write $\Psi(Q)= \sum_{q=1}^{Q} \phi_{q}$, and suppose that $\Psi(Q) \rightarrow \infty$ and $Q \rightarrow \infty$. Let $C>0$ be arbitrary, and suppose that for every $n\in\mathbb{N}$,
\begin{equation}\label{Aist int2}
    \int_{X}\left(\sum_{1\leq q \leq n}\left(f_{q}(x)-f_{q}\right)\right)^{2} \textrm{d}\mu = O_C\left(\frac{\Psi(Q)^{2}}{\left(\log \Psi(Q)\right)^{C}}\right)
\end{equation} 
Then for any $\delta > 0$, there exists an $E_{C,\,\delta} \subset X$ such that $\mu (E_{C,\,\delta}) < \delta$ and for any $x \in X \setminus E_{C,\,\delta}$, 
 \[\sum_{q=1}^{Q} f_{q}(x) \leq \sum_{q=1}^{Q}f_{q}+\max\left\{ k_{C,\,\delta}K,\,2\frac{e\Psi\left(Q\right) +K }{\left(\log \Psi\left(Q\right)\right)^{C}}+K\right\},\]
where $\mathcal{C}=\mathcal{C}(C)$ is a constant obtained by finding an explicit bound in \eqref{Aist int2} for
\begin{align} \label{k_EABHT}
    k_{C,\,\delta} \coloneqq\min \left\{k\,:\,\mathcal{C}k^{-\frac{\sqrt{C}}{2}}\left(1+\frac{2k}{\sqrt{C}-2}\right)<\delta\right\},
 \end{align}

 and $\mathcal{C}$ is a constant obtained by finding an explicit bound in \eqref{Aist int2}. 
\end{thm}

A bound for the constant $\mathcal{C}$ (which is the implicit constant in \eqref{Aist int2}) has been established and will be published separately due to the length of the calculation.

We now give the proof of Theorem \ref{effective solution counting} assuming Theorem \ref{effective v3}.

\begin{proof}
The proof follows that from \cite{aistleitner2022metric}, using Theorem \ref{effective v3} in place of Theorem 1 from \cite{aistleitner2022metric}. First define 

\begin{align*}
    \mathcal{A}_{q}\coloneqq[0,1]\bigcap\bigcup_{\substack{1 \leq p \leq q \\ \textrm{gcd}(p,\,q)=1}}\left(\frac{p-\psi(q)}{q},\, \frac{p+\psi(q)}{q}\right),\quad q\in\mathbb{N}
\end{align*}
We now note that if $\mu$ is the Lebesgue measure on $\left[0,\,1\right]$, then
\[\mu\left(\mathcal{A}_{q}\right)=\frac{2 \varphi(q)\psi(q)}{q}.\]

In Theorem \ref{effective v3}, we set $f_{q}=\phi_{q} =\mu\left(\mathcal{A}_{q}\right)$ and set $f_{q}(x)= \mathds{1}_{\mathcal{A}_{q}}(x)$, where $\mathds{1}_{\mathcal{A}_{q}}$ denotes the characteristic function on $\mathcal{A}_{q}$. We note that this means that $\Psi(Q)$ as defined in Theorem \ref{effective solution counting} and Theorem \ref{effective v3} are equivalent.

We now note that
\[\int_{0}^{1}\left(\sum_{q=1}^{Q}\left(\mathds{1}_{\mathcal{A}_{q}}(x)-\mu\left(\mathcal{A}_{q}\right)\right)\right)^{2}\mathrm{d}x=\sum_{q,\,r \leq Q}\mu\left(\mathcal{A}_{q} \cap \mathcal{A}_{r}\right)-\Psi(Q)^{2}.\]
By Theorem 2 from \cite{aistleitner2022metric}, we have that
\[\sum_{q,\,r \leq Q}\mu\left(\mathcal{A}_{q} \cap \mathcal{A}_{r}\right)-\Psi(Q)^{2} = O_{C}\left(\frac{\Psi(Q)^{2}}{\left(\log \Psi(Q)\right)^{C}}\right),\]
with an implied constant depending on $C$ for arbitrary $C>0$. We note that this implied constant is the constant that has been calculated, with the working to be published separately due to length. We can thus apply Theorem \ref{effective v3} and the statement follows. 
\end{proof}

\subsection{Inhomogeneous Diophantine Approximation on $M_{0}$-sets}
Theorems 2 and 3 can also be applied to give effective results akin to those in \cite{Pollington2022}, which is about inhomogeneous Diophantine approximation on $M_0$-sets. For background and motivation we refer the reader to the introduction and further discussion in \cite{Pollington2022}; however, we begin this section by recalling some definitions to make this paper self contained. 

As usual, the Fourier transform of a non-atomic probability measure $\mu$ is defined by
\[\hat{\mu}(t)=\int e^{-2\pi i t x}\textrm{d}\mu(x),\, t \in \mathbb{R}.\]
A closed set $E\subset \mathbb{R}$ is said to be an $M_0$-set if there exists a probability measure $\mu$ on $E$ such that $\hat\mu$ vanishes at infinity.

Given an increasing sequence of natural numbers $\mathcal{A}=(q_{n})_{n \in \mathbb{N}}$, we call $\mathcal{A}$ lacunary if there exists a constant $K > 1$ such that for all $n \in \mathbb{N}$,
\begin{equation}\label{eq17}
 \frac{q_{n+1}}{q_{n}}\geq K.  
\end{equation}

Relatedly, let $\alpha \in (0,1).$ We say that a sequence of increasing natural numbers $\mathcal{A}$ is $\alpha$-separated if there exists a constant $m_{0} \in \mathbb{N}$ such that for any integers $m_{0}\leq m <n$, we have that if
\[1 \leq \left|sq_{m}-tq_{n}\right|<{q_{m}}^{\alpha}\]
for some $s,\,t \in \mathbb{N}$, then
\[s>m^{12}.\]

Finally, let $\psi:\mathbb{N} \to \mathbb{R}^+$ be a real, positive function and let $\gamma \in [0,1]$. We define the counting function
\begin{equation}\label{R(x,N)}
    R(x,\,N)=R(x,\,N;\,\gamma,\, \psi,\, \mathcal{A})=\#\left\{1 \leq n \leq N\,:\, \left\|q_{n}x-\gamma \right\|\leq \psi(q_{n})\right\},    
\end{equation}
where $\left\|\alpha\right\|:=\min\left\{\left|\alpha-m\right|: m \in \mathbb{Z}\right\}$ denotes the absolute distance from $\alpha$ to its nearest integer.

We will now state effective versions of Theorems 1 and 4 from \cite{Pollington2022}; for their original asymptotic versions, we refer the reader to the original paper.



\begin{thm}[Effective Theorem 1 from \cite{Pollington2022}]\label{E.Thm1}
Let $F\subset[0,1]$ and $\mu$ be an non-atomic probability measure supported on $F$. Let $\gamma\in[0,1]$ and $\psi:\mathbb{N}\to(0,1]$ be a function. Let $\mathcal{A}=(q_n)_{n\in\mathbb{N}}$ be a lacunary sequence of positive integers. Suppose there exists $\nu >0$ and $A>2$ such that for any $t\in\mathbb{R}\setminus[-1,1]$, \begin{align}
\label{eq8}
    \left|\hat{\mu}(t)\right|\leq\frac{\nu }{\log^A{|t|}}
. \end{align}
Then for any $\varepsilon>0$ and $\delta>0$, there exists $\mu$-measurable $E_{\varepsilon,\delta}\subset F$ such that $\mu(E_{\varepsilon,\delta})<\delta$ and for any $x\in F\setminus E_{\varepsilon,\delta}$ and $N\in\mathbb{N}$, \begin{align}
    \label{E.T1}
        |R(x,N)-2\Psi(N)|\leq 2K_{\varepsilon,\delta/2}\left(\Psi(N)^{2/3}\left(\log{\Psi(N)}+2\right)^{2+\varepsilon}\right)
    +t_{1,\delta}
    ,\end{align} 
    where $\alpha\in(0,1)$, $\log^{+}(\beta)\coloneqq\max{(0,\log \beta)}$ and
\begin{align}
    \label{Psi}
    \Psi(n)&=\sum_{k=1}^n\psi(k)\\
    \label{phithm1}
    \varPhi(N)&=\Psi^{4/3}(\log^+{\varPhi(N)+1})+\Psi(N), \\
    \label{Kthm1}
    K&=6\max{(48,c_1,c_2)}, \\
    \label{c1}
    c_1&= \frac{22}{K_0-1}, \\
    \label{c2}
    \begin{split}
        c_2     &= 12\left(\frac{3}{2^{2/3}}+1\right)\left(1+\zeta{\left(A-1\right)}\right)+\frac{8}{K_0-1} \\
                &\qquad +18\nu C^{-A} \left(\sqrt{2}\zeta{\left(A-\frac{1}{2}\right)}\left(2+\alpha^{-A}\right)+2^{A+1}\zeta{\left(\frac{A}{2}\right)}\right),
    \end{split}\\
    \label{t,1,delta}
    t_{1,\delta}&=\frac{1}{2}+\left(\frac{(1-A)\delta}{2(3+\nu/C^A)}\right)^{1/(1-A)},
\end{align}
where $N_{\varepsilon,\delta}$, $r_{\varepsilon,\delta}$, and $K_{\varepsilon,\delta}$ and $\varPhi_{0}$ are given in Theorem \ref{thm3}.
\end{thm}

\begin{thm}[Effective Theorem 4 from \cite{Pollington2022}]\label{E.Thm4}
Let $F\subset[0,1]$ and $\mu$ be an non-atomic probability measure supported on $F$. Let $\gamma\in[0,1]$ and $\psi:\mathbb{N}\to(0,1]$ be a function. Let $\mathcal{A}=(q_n)_{n\in\mathbb{N}}$ be an $\alpha$-separated sequence. Suppose there exists $\nu >0$ and $A>2$ such that for any $t\in\mathbb{R}\setminus[-1,1]$, \eqref{eq8} is satisfied.  Further, assume that $\mathcal{A}=(q_{n})_{n \in \mathbb{N}}$ satisfies the following growth condition:
\begin{equation}\label{eq32}
\log q_{n}> Cn^{1/B}    
\end{equation}

Then for any $\varepsilon>0$ and $\delta>0$, there exists $\mu$-measurable $E_{\varepsilon,\delta}\subset F$ such that $\mu(E_{\varepsilon,\delta})<\delta$ and for any $x\in F\setminus E_{\varepsilon,\delta}$ and $N\in\mathbb{N}$, \begin{align}
    \label{E.T4}
        \begin{split}
        |R(x,N)-2\Psi(N)|&\leq K_{\varepsilon,\delta/2}\left(\left(\Psi(N)\left(\log^+{\Psi(N)}+2\right)+E(N)\right)^{1/2}\right. \\
                &\qquad \left.\left(\log{\left(\Psi(N)\left(\log^+{\Psi(N)}+2\right)+E(N)\right)}\right)^{3/2+\varepsilon}+2\right)
        +t_{2,\delta},
        \end{split}
    \end{align}
    where $R(x,N)$, $\Psi(N)$ and $c_2$ are given at \eqref{R(x,N)}, \eqref{Psi} and \eqref{c2} respectively. Further,
\begin{align}
E(N)&=\mathop{\sum\sum}_{1 \leq m<n\leq N} (q_{m},\,q_{n})\min \left(\frac{\psi(q_{m})}{q_{m}},\, \frac{\psi(q_{n})}{q_{n}}\right), \nonumber \label{def of E(N)}\\
\varPhi(N)&=\Psi(N)\left(\log^+{\Psi(N)}+2\right)+E(N), \\
    \label{Kthm4}
    K&=2\max{(48,c_3)}, \\
    \label{c3}
    c_3&=4+18\nu C^{-A} \left(\sqrt{2}\zeta{\left(\frac{A}{B}-\frac{1}{2}\right)}\left(2+\alpha^{-A}\right)+2^{A+1}\zeta{\left(\frac{A}{2B}\right)}\right)+c_2, \\
    \label{t,2,delta}
    t_{2,\delta}&=\frac{1}{2}+\left(\frac{(1-\min{(9,A/B)})\delta}{2(1+\nu/C^A)}\right)^{1/(1-\min{(9,A/B)}))},
\end{align}
where $N_{\varepsilon,\delta}$, $r_{\varepsilon,\delta}$, and $K_{\varepsilon,\delta}$ and $\varPhi_{0}$ are given in Theorem \ref{thm3}.
\end{thm}

The proofs of these follow as in previous sections; various lemmas are made fully effective before applying Theorem 3. However, the lemmas needed to make the results fully effective rely heavily on the original proofs given in \cite{Pollington2022}. In this section we state the effective lemmas needed and prove the theorems above using these; proofs of the effective lemmas are left to the appendix.

We begin by noting that $t_{1,\delta}$ and $t_{2,\delta}$ appeared in \eqref{t,1,delta} and \eqref{t,2,delta} are needed to account for some assumptions made during the proofs, namely \eqref{extra con L5} and \eqref{eq91} below. In our proofs, we will initially make these assumptions, before proving Lemma \ref{lastly count update additive constant thingy}, which allows us to remove them. We further note that a lacunary sequence, as discussed in Theorem 6, satisfies \eqref{eq32} with $B=1$; thus in the following, when proving Theorem 6 we take $B=1$.

We note that if $\mu$ is the Lebesgue measure on $[0,1]$, then $\mu$ is a probability measure and \eqref{eq8} is satisfied by taking $\nu =1/\pi$.

In what follows we give effective versions of results from \cite{Pollington2022}, before going on to prove Theorems \ref{E.Thm1} and \ref{E.Thm4} above. Specifically, we make Lemmas 5 to 8 and Propositions 1 and 2 from \cite{Pollington2022} effective, then apply the results to Theorem 1 and 4 in that paper to give the quantitative versions. We begin by giving statements of explicit versions of Lemmas 5 to 8 from \cite{Pollington2022}, then use these to make Propositions 1 and 2 from \cite{Pollington2022} effective. We give the proofs of these results in the appendix. 

\begin{lemma}[Effective Lemma 5 from \cite{Pollington2022}]\label{E.Lemma5} 
    Let $\mu$ be a non-atomic probability measure supported on $F\subset[0,1]$ and $(q_n)_{n\in\mathbb{N}}$ be an increasing sequence of positive integers greater than 4. Suppose there exists $B\geq1$ and $C>0$ such that growth condition \eqref{eq32} is satisfied. Let $\gamma\in[0,1]$ and $\psi:\mathbb{N}\to(0,1]$ be a function. Suppose there exists $A>2B$ such that \eqref{eq8} is satisfied. Further suppose that for any $\tau>1$ and $n\in\mathbb{N}$,
    \begin{align}
        \label{extra con L5}
            \psi(q_n)\geq3n^{-\tau}.
     \end{align}
 
    It then follows that for arbitrary $a,b\in\mathbb{N}$, if $a<b$ then we have that
\begin{align}
    \label{E.L5}
        \left|\sum_{n=a}^b\mu\left(E_{q_n}^\gamma\right)-2\sum_{n=a}^b\psi(q_n)\right|\leq\min{\left(m_1,m_2\sum_{n=a}^b\psi(q_n)\right)}
    , \end{align} 
    where for any $q\in\mathbb{N}$, \begin{align}
        E_q^\gamma=E_q^\gamma(\psi)&\coloneqq\{x\in[0,1]:\|qx-\gamma\|\leq\psi(q)\}, \\
        \label{m1m2}
        m_1&=m_1(A,B)=3+3\zeta{\left(\frac{A}{B}-1\right)}, \\
        m_2&=\frac{3}{2^{2/3}}<2
    . \end{align}
\end{lemma}

As commented previously, the assumption \eqref{extra con L5} does not appear in Theorem \ref{E.Thm1} or \ref{E.Thm4};  this issue is handled in Lemma \ref{lastly count update additive constant thingy}.

We now give a bound for the value
\[S(m,n)\coloneqq \sum_{k \in \mathbb{Z} \backslash \left\{0\right\}} \Hat{W}^+_{m,n}(k)\Hat{\mu}(-k),\]
where 
\begin{equation}\label{W defn}
 W_{m,\,n}^{+}(k)=\left(\left(\sum_{p=0}^{q_{m}-1}\delta_{\frac{p+\gamma}{q_{m}}}(k)\right)*\chi^{+}_{\frac{\psi(q_{m})}{q_{m}},\varepsilon_{m}}(k)\right)\left(\left(\sum_{r=0}^{q_{n}-1}\delta_{\frac{r+\gamma}{q_{n}}}(k)\right)*\chi^{+}_{\frac{\psi(q_{n})}{q_{n}},\varepsilon_{n}}(k)\right),   
\end{equation}

where $*$ denotes convolution, $\delta_{x}$ is the Dirac delta-function at the point $x \in \mathbb{R}$ and
\begin{equation*}
    \chi_{\delta,\varepsilon}^{+}(x)\coloneqq 
    \begin{cases}
        1 & \text{if } \lvert \lvert x\rvert\rvert \leq \delta \\
    \frac{1}{\delta \varepsilon}(\delta-\lvert\lvert x\rvert\rvert)              & \text{if } (1-\varepsilon)\delta < \lvert\lvert x\rvert\rvert \leq \delta \\
    0 & \text{if } \lvert\lvert x \rvert\rvert > \delta,
    \end{cases}
\end{equation*}
where $\lvert \lvert x \rvert \rvert$ denotes the distance from $x \in \mathbb{R}$ to the nearest integer; that is, $\lvert\lvert x \rvert\rvert\coloneqq \min \{\lvert x-m\rvert\,:\,m \in \mathbb{Z}\}$.

\begin{lemma}[Effective Lemma 6 from \cite{Pollington2022}]\label{E.Lemma6} 
    Let $\mu$ be a probability measure supported on $F\subset[0,1]$. Let $\mathcal{A}=(q_n)_{n\in\mathbb{N}}$ be an increasing sequence of positive integers greater than 4. Suppose there exists $B\geq1$ and $C>0$ such that growth condition \eqref{eq32} is satisfied. Let $\gamma\in[0,1]$, $\psi:\mathbb{N}\to(0,1]$, let $(\varepsilon_n)_{n\in\mathbb{N}}$ be a sequence of real numbers in $(0,1]$ and let $\alpha \in (0,\,1)$. Suppose there exists $A>2B$ such that \eqref{eq8} is satisfied. Then for any $m,n\in\mathbb{N}$, if $m<n$ then 
    \begin{align}
    \label{E.L6}
        |S(m,n)|\leq9\nu C^{-A} \frac{\psi(q_m)}{n^{A/B}{\varepsilon_n}^{1/2}}+9\nu C^{-A} \left(1+\frac{1}{\alpha^A}\right)\frac{\psi(q_n)}{m^{A/B}{\varepsilon_m}^{1/2}}+\frac{9\cdot 2^A\nu C^{-A} }{n^{A/B}{\varepsilon_m}^{1/2}{\varepsilon_n}^{1/2}}+|T(m,n)|,
    \end{align} 
    where $T$ is defined based on the value of  $\alpha$ that
    \begin{equation}\label{Tmn}
        T(m,n)=T_{\alpha}(m,n)\coloneqq\mathop{\sum\sum}_{\substack{s,\,t \in \mathbb{Z}\backslash\{0\} \\ 1 \leq \lvert sq_{m}-tq_{n}\rvert<{q_{m}}^{\alpha}}}\Hat{W}^{+}_{q_{m},\gamma,\epsilon_{m}}(sq_{m})\Hat{W}^{+}_{q_{n},\gamma,\epsilon_{n}}(tq_{n})\Hat{\mu}(sq_{m}-tq_{n}),
    \end{equation}
    and $\Hat{W}^{+}_{q_{m},\gamma,\epsilon_{m}}(sq_{m})$ is the fourier transform of \eqref{W defn}.
\end{lemma}

The following two lemmas give effective bounds on the size of the quantity $T(m,\,n)$ appearing above. Lemma \ref{E.Lemma7} deals with lacunary sequences and Lemma \ref{E.Lemma8} deals with $\alpha$-separated sequences.

\begin{lemma}[Effective Lemma 7 from \cite{Pollington2022}]\label{E.Lemma7} 
    Let $\mu$ be a probability measure supported on $F\subset[0,1]$. Let $\mathcal{A}=(q_n)_{n\in\mathbb{N}}$ be an lacunary sequence of natural numbers, with the constant from the definition at \eqref{eq17} given by $K=K_0$. Let $\gamma\in[0,1],\,\alpha\in(0,1),\,\psi:\mathbb{N}\to(0,1]$ and $(\varepsilon_n)_{n\in\mathbb{N}}$ be a decreasing sequence of real numbers in $(0,1]$. Then for arbitrary $a,b\in\mathbb{N}$ with $a<b$, we have that
    \begin{align*}
        \mathop{\sum\sum}_{a\leq m<n\leq b}|T(m,n)|\leq11K'\sum_{n=a}^b\frac{\psi(q_n)}{{\varepsilon_n}^{1/2}},
    \end{align*}
    where $T(m,\,n)$ depends on $\alpha$ and it is given by \eqref{Tmn}, and $K'$ is given by \begin{align}
    \label{K'}
        K'=\frac{1}{K_0-1}>0.
     \end{align}
  \end{lemma}
    
\begin{lemma}[Effective Lemma 8 from \cite{Pollington2022}]\label{E.Lemma8} 
    Let $\mu$ be a probability measure supported on $F\subset[0,1]$. Let $\mathcal{A}=(q_n)_{n\in\mathbb{N}}$ be an $\alpha$-separated increasing sequence with $m_0=1$. Let $\gamma\in[0,1]$, $\psi:\mathbb{N}\to(0,1]$ and $(\varepsilon_n)_{n\in\mathbb{N}}$ be a sequence of real numbers in $(0,1]$. Suppose that for any $n\in\mathbb{N}$, \begin{align}
        \label{eq91}
        \psi(q_n)&\geq n^{-9}
    \end{align} and ${\varepsilon_n}^{-1}\leq2n$. Then for arbitrary $a,b\in\mathbb{N}$, if $a<b$ then 
    \begin{align*}
        \mathop{\sum\sum}_{a\leq m <n\leq b}|T(m,n)|\leq2\sum_{n=a}^b\psi(q_n),
    \end{align*}
    where $T(m,\,n)$ is given by \eqref{Tmn}.
\end{lemma}

As commented previously, the assumption of \eqref{eq91} does not appear in Theorem \ref{E.Thm1} or \ref{E.Thm4},  this issue is handled in Lemma \ref{lastly count update additive constant thingy}.

To prove the main theorems, we need effective versions of Propositions 1 and 2 from \cite{Pollington2022}. We state these here and leave the proofs to the appendix. 

 \begin{prop}[Effective Proposition 1 from \cite{Pollington2022}]\label{E.Prop1} 
Let $F,\, \mu,\, \mathcal{A}=(q_{n})_{n \in \mathbb{N}},\, \gamma$ and $\psi$ be as in Theorem \ref{E.Thm1}. Further, assume that $\psi$ satisfies \eqref{extra con L5}. Then, for arbitrary $a,\,b \in \mathbb{N}$, with $a<b$ we have that
\begin{align*}
    2\mathop{\sum\sum}_{a\leq m<n\leq b}\mu(E_{q_m}^\gamma\cap E_{q_n}^\gamma)
    &\leq\left(\sum_{n=a}^b\mu(E_{q_n}^\gamma)\right)^2+48\left(\sum_{n=a}^b\psi(q_n)\right)^{4/3}\log^+{\left(\sum_{n=a}^b\psi(q_n)\right)}\\
    &+c_1\left(\sum_{n=a}^b\psi(q_n)\right)^{4/3}+c_2\sum_{n=a}^b\psi(q_n)
    , \end{align*}
    where $c_1$ and $c_2$ are given in \eqref{c1} and \eqref{c2} respectively. 
\end{prop}

\begin{prop}[Effective Proposition 2 from \cite{Pollington2022}]\label{E.Prop2} 
Let $F,\,\mu,\,\mathcal{A}=(q_{n})_{n\in \mathbb{N}},\,\gamma$ and $\psi$ be as in Theorem \ref{E.Thm4}. Further, assume $\psi$ satisfies \eqref{extra con L5}, $q_{1}>4$ and that $\mathcal{A}$ is $\alpha$-separated with the implicit constant $m_{0}=1$. Then, for arbitrary $a,\,b \in \mathbb{N}$ with $a<b$, we have that
    \begin{align*}
        2\mathop{\sum\sum}_{a\leq m<n\leq b}\mu(E_{q_m}^\gamma\cap E_{q_n}^\gamma)
        \leq \left(\sum_{n=a}^b\mu(E_{q_n}^\gamma)\right)^2
         +8\mathop{\sum\sum}_{a\leq m<n\leq b}\gcd{(q_m,q_n)}\min{\left(\frac{\psi(q_m)}{q_m},\frac{\psi(q_n)}{q_n}\right)} \\
         +48\left(\sum_{n=a}^b\psi(q_n)\right)\log^+{\left(\sum_{n=a}^b\psi(q_n)\right)}
         +c_3\sum_{n=a}^b\psi(q_n)
    , \end{align*}
    where $c_3$ is given in \eqref{c3}.
\end{prop}

Finally, we can make use of Theorem \ref{thm3} to make Theorems 1 and 4 from \cite{Pollington2022} effective, modulo the error term from assumptions \eqref{extra con L5} and \eqref{eq91} which we make effective immediately after. We give these proofs in full here to demonstrate the use of Theorem \ref{thm3}. For Theorem 1 from \cite{Pollington2022}, we take the following parameter in Theorem \ref{thm3}: 
\begin{align*}
    \varphi_n&=\psi(q_n)\Psi(n)^{1/3}\left(\log^+\Psi(n)+1\right)+2\psi(q_n).
\end{align*}

Notice that for any $\varepsilon>0$ and $N\in\mathbb{N}$, \begin{align*}
    \varPhi^{1/2}(N)\log^{3/2+\varepsilon}{\varPhi(N)}+2\leq2\Psi^{2/3}(N)(\log{\Psi(N)+2})^{2+\varepsilon}
, \end{align*}
where $\varPhi$ is as given in the original proof or in \eqref{phithm1}. This proves Theorem \ref{E.Thm1}.

For Theorem 4 from \cite{Pollington2022}, we set parameters of Theorem \ref{thm3} as follows:
\begin{align*}
    \varphi_n&=\psi(q_n)\left(\log^+\Psi(n)+2\right)+\sum_{m=1}^{n-1}\gcd{(q_m,q_n)}\min{\left(\frac{\psi(q_m)}{q_m},\frac{\psi(q_n)}{q_n}\right)}, \\
    \varPhi(N)&=\sum_{n=1}^N\varphi_n\leq \Psi(N)\left(\log^+{\Psi(N)}+2\right)+E(N)
, \end{align*}
where \begin{align*}
    E(N)&=\sum_{m=1}^{N-1}\sum_{n=m+1}^N \gcd{(q_m,q_n)}\min{\left(\frac{\psi(q_m)}{q_m},\frac{\psi(q_n)}{q_n}\right)}
. \end{align*} 
Theorem \ref{E.Thm4} follows.

We now consider the final, extra term in each each effective theorem, which allow us to remove the two extra assumptions in Lemma \ref{E.Lemma5} and Lemma \ref{E.Lemma8}, namely conditions \eqref{extra con L5} and \eqref{eq91}. In the following lemma, we consider this extra term.

\begin{lemma}\label{lastly count update additive constant thingy}
    Let $\psi:\mathbb{N}\to(0,1]$ be a function and $\mathcal{A}=(q_n)_n$ be an increasing sequence of positive integers. Suppose there exists $\nu >0$ and $A>2$ such that for any $t\in\mathbb{R}\setminus[-1,1]$, \eqref{eq8} is satisfied. Suppose there exists $B\geq1$ and $C>0$ such that $A>2B$ and \eqref{eq32} is satisfied. Suppose $\Psi(n)=\sum_{k=1}^n\psi(k)$ is unbounded. Let $\omega:\mathbb{N}\to[0,+\infty)$ be a function. Suppose \begin{align*}
        \sum_{n=1}^\infty\omega(n)<+\infty.
    \end{align*} Define an auxiliary function $\psi^*(q_n)=\max{(\psi(q_n),\omega(n))}$. Then for every $\delta>0$ there exists $F_{\delta}\subset F$ such that $\mu(F_{\delta})<\delta$ and for any $x\in F\setminus F_{\delta}$ and $N\in\mathbb{N}$, \begin{align*}
        \left|R(x,N;\gamma,\psi, \mathcal{A})-R(x,N;\gamma,\psi^*,\mathcal{A})\right|\leq t_{\delta}
    , \end{align*} 
    where \begin{align}
        \label{t_delta}
        t_{\delta}=\min{\left\{t\in\mathbb{N}:\sum_{n=t}^\infty
        \left(
        \omega(n)+\frac{\nu}{C^A n^{A/B}}
        \right)<\frac{\delta}{3}\right\}}
    . \end{align}
\end{lemma}
\begin{proof}
    By the definition of the counting function \eqref{R(x,N)}, for any $x\in F$, we have that
    \begin{align*}
        R(x,N;\gamma,\psi,\mathcal{A})
        \leq R(x,N;\gamma,\psi^*,\mathcal{A})
        \leq
        R(x,N;\gamma,\psi,\mathcal{A})
        + R(x,N;\gamma,\omega,\mathcal{A}).
    \end{align*}
    It follows that, for any $x\in F$,
    \begin{align*}
        |R(x,N;\gamma,\psi,\mathcal{A})
        - R(x,N;\gamma,\psi^*,\mathcal{A})|
        \leq R(x,N;\gamma,\omega,\mathcal{A}).
    \end{align*}
    It suffices to find an upper for $R(x,N;\gamma,\omega,\mathcal{A})$, for any $x\in F\setminus F_\delta$, for some measurable $F_\delta\subset F$ such that $\mu(F_\delta)<\delta$.
    By Theorem 2 from \cite{Pollington2022},
    we see that for almost every $x\in F$, the extra term is exactly given by the counting function
    \begin{align*}
        R(x,N;\gamma,\omega,\mathcal{A})
    \end{align*}
    For any $q\in\mathbb{N}$, define
    \begin{align*}
        E_q=\{x\in F:\|qx-\gamma\|\leq\omega(q)\}
    .\end{align*}
    By Lemma 2 from \cite{Pollington2022}, we know that for any $q\in\mathbb{N}$, if $q\geq4$ then \begin{align*}
        \mu(E_q)\leq 3\omega(q)+\min{\left\{3\max_{s\in\mathbb{Z}}\left|\hat{\mu}(sq)\right|,2\sum_{s=1}^\infty\frac{\left|\hat{\mu}(sq)\right|}{s}\right\}}.
    \end{align*}
    By the assumptions of the lemma,
    \begin{align*}
        \max_{s\in\mathbb{Z}\setminus\{0\}}\left|\hat{\mu}(sq_n)\right|
        \leq\frac{\nu}{\log^A{q_n}}
        \leq\frac{\nu}{C^A n^{A/B}}.
    \end{align*}
    Hence, by taking the summation, for any $t\in\mathbb{N}$, we have that
    \begin{align*}
        \mu\left(\bigcup_{n=t}^\infty E_{q_n}\right)
        \leq\sum_{n=t}^\infty{\mu(E_{q_n})}
        \leq3\sum_{n=t}^\infty{w(n)}+\frac{3\nu}{C^A}\sum_{n=t}^\infty\frac{1}{n^{A/B}}
        \leq3\sum_{n=t}^\infty
        \left(
        \omega(n)+\frac{\nu}{C^A n^{A/B}}
        \right)
        .
    \end{align*}
    As the right-most term converges when $t=1$, we get that for any $\delta>0$, there exists $t_\delta\in\mathbb{N}$ such that \begin{align*}
        \mu\left(\bigcup_{n=t_\delta}^\infty E_{q_n}\right)\leq\delta,
    \end{align*}
    where $t_\delta$ is given by \eqref{t_delta}.
    By taking $F_\delta=\bigcup_{n=t_\delta}E_{q_n}$, we get that $\mu(F_\delta)<\delta$ and for any $x\in F\setminus F_\delta$, \begin{align*}
        R(x,N;\gamma,\omega,\mathcal{A})
        <t_\delta.
    \end{align*}
\end{proof}

This lemma tells us that it is possible to make two such extra assumptions on $\psi$ and the counting result differs by an additive constant which depends only on $\omega$. In the proof for Theorem \ref{E.Thm1}, we have taken $\tau=A/B$ in \eqref{extra con L5}. That is, for Theorem \ref{E.Thm1}, $t_\delta$ is given by \begin{align*}
    t_\delta=\min{\left\{t\in\mathbb{N}:\sum_{n=t}^\infty
        \left(3n^{-A/B}+\frac{\nu}{C^A n^{A/B}}
        \right)<\frac{\delta}{3}\right\}}
. \end{align*}
To get a concrete estimate, an upper bound for this $t_\delta$ can be obtained by noticing that if
\begin{align*}
    \left(3+\frac{\nu}{C^A}\right)\sum_{n=t}^\infty\frac{1}{n^{A/B}}
    \leq\left(3+\frac{\nu}{C^A}\right)\int_{t-1/2}^\infty\frac{dx}{x^{A/B}}
    =\frac{3+\nu/C^A}{1-A/B}\left(t-\frac{1}{2}\right)^{1-A/B}<\delta
, \end{align*}
then
\begin{align*}
    t_{\delta}\leq\frac{1}{2}+\left(\frac{(1-A/B)\delta}{3+\nu/C^A}\right)^{1/(1-A/B)}.
\end{align*}
Also, for Theorem \ref{E.Thm4}, $t_\delta$ is given by \begin{align*}
    t_\delta=\min{\left\{t\in\mathbb{N}:\sum_{n=t}^\infty
        \left(n^{-9}+\frac{\nu}{C^A n^{A/B}}
        \right)<\frac{\delta}{3}\right\}}
. \end{align*}
To get a concrete estimate, an upper bound for this $t_\delta$ can be obtained by noticing that if $d=\min{(9,A/B)}$,
\begin{align*}
    \sum_{n=t}^\infty\left(n^{-9}+\frac{\nu}{C^A n^{A/B}}\right)
    \leq \left(1+\frac{\nu}{C^A}\right)\sum_{n=t}^\infty\frac{1}{n^d}
    \leq \left(1+\frac{\nu}{C^A}\right)\int_{t-1/2}^\infty\frac{dx}{x^d}
    =\frac{1+\nu/C^A}{1-d}\left(t-\frac{1}{2}\right)^{1-d}<\delta
, \end{align*}
then
\begin{align*}
    t_{\delta}\leq\frac{1}{2}+\left(\frac{(1-\min{(9,A/B)})\delta}{1+\nu/C^A}\right)^{1/(1-\min{(9,A/B)}))}.
\end{align*}

Our new respective set $E_{\varepsilon,\delta}$ for Theorems \ref{E.Thm1} and \ref{E.Thm4} is given by $E_{\varepsilon,\delta}=E'_{\varepsilon, \delta/2}\cup F_{\delta/2}$, where $E'_{\varepsilon, \delta/2}$ is given by Theorem \ref{thm3} with the parameters given above, and $F_{\delta/2}$ is given in Lemma \ref{lastly count update additive constant thingy}. This completes the proofs of Theorems \ref{E.Thm1} and \ref{E.Thm4}.

\subsection{Normal Numbers}
We begin by recalling some definitions. Set $b \in \mathbb{N},\, b \geq 2$. Recall that for any real number $\alpha$ there is a unique expansion base $b$ such that
\begin{equation}\label{base b expansion defn}
 \alpha=\left[\alpha\right]+\sum_{n=1}^{\infty}a_{n}b^{-n},   
\end{equation}
where $\left[\alpha\right]$ denotes the integer part of $\alpha$, $0\leq a_{n} <b$ and $a_{n} < b-1$ infinitely often. Given a fixed $\alpha$, denote by $A(d,\,b,\,N)$ the number of times $d$ is in the set $\left\{a_{1},\,\dots,\,a_{N}\right\}$. We say that $\alpha$ is simply normal to base $b$ if 
\[\lim_{N \to \infty}\frac{A(d,\,b,\,N)}{N}=\frac{1}{b}\]
for all $d,\, 0 \leq d <b$.

We call $\alpha$ entirely normal to base $b$ if it is simply normal base $b^{n}$ for all $n=1,\,2,\dots$. We say $\alpha$ is absolutely normal if it is entirely normal to all bases $b >1$.

These definitions are different to those first given by Borel; Theorem 1.2 from \cite{harman1998metric} shows the definitions are equivalent.

An easy application of Lemma 1.4 from \cite{harman1998metric} quickly shows that almost all real numbers are simply normal to a base $b$. Applying Theorem \ref{thm2} in its place allows us to give an upper bound on the number of times a given digit $d$ appears in the base $b$ expansion for almost all real $\alpha$ other than in a set of measure at most $\delta$.

\begin{thm}\label{normal numbers thm}
For any $\delta>0$, there exists a set $E_{\delta}$ of measure at most $\delta$ such that for any real number $\alpha\in[0,1)\setminus E_{\delta}$, the number of times a given digit $d$ appears in its base $b$ expansion in \eqref{base b expansion defn} up to the $N$-th digit (that is, $A(d,\,b,\,N)$) satisfies
\[A(d,\,b,\,N) \leq \min \left\{N,\, \frac{N}{b}+K_{\varepsilon,\,\delta} \left(N^{2/3} \log^{1/3+\varepsilon}\left(N+2\right)\right)\right\},\]
where $K_{\varepsilon,\,\delta}$ is given in Theorem \ref{thm2}.
\end{thm}
We note that the size of the constant $K_{\varepsilon,\,\delta}$ impacts the size we need $N$ to be for 
\[\frac{N}{b}+K_{\varepsilon,\,\delta} \left(N^{2/3} \log^{1/3+\varepsilon}\left(N+2\right)\right) < N\]
to hold; $N$ is clearly a trivial upper bound for $A(d,\,b,\,N)$.

The proof follows the one given in \cite{harman1998metric} but replaces the use of the ineffective lemma with the effective version given at Theorem \ref{thm2}.

\begin{proof}
As the integer part of $\alpha$ has no bearing on whether $\alpha$ is simply normal base $b$, we can without loss of generality restrict $\alpha \in \left[0,\,1\right)$. Further, let $a_{k}$ denote the $k$-th digit in the base $b$ expansion of $\alpha$ as given at \eqref{base b expansion defn}. Set $d \in \mathbb{Z},\, 0\leq d <b$.

Let
\begin{equation*}
    f_{k}(\alpha)= \begin{cases}
    1 \textrm{ if the $k$-th digit of $\alpha$ is $d$,}\\
    0 \textrm{ otherwise.}
    \end{cases}
\end{equation*}
Further, let
\[f_{k}=b^{-1}.\]

We note that for $j \neq k$,
\[\int_{0}^{1}f_{k}(x)f_{j}(x) \textrm{d}x= \mu \left(\left\{x \in \left[0,\,1\right): \textrm{the $k$-th and $j$-th digits of $x$ are both $d$}\right\}\right)=b^{-2}.\]

It thus follows that
\[\int_{0}^{1}\left(\sum_{k=1}^{N}\left(f_{k}(x)-f_{k}\right)\right)^{2} \textrm{d}x = \sum_{k=1}^{N}b^{-1}\left(1-b^{-1}\right).\]
More justification for these equalities can be found on page 12 from \cite{harman1998metric}.

It follows that we can apply Theorem \ref{thm2} with $\varphi_{k}=b^{-1}$ and $K=1$. We note that $\Phi_{N}\leq \frac{N}{b},\, \sum_{k=1}^{N}f_{k}\leq \frac{N}{b}$ and the result then follows.
\end{proof}

\subsection{Strong Law of Large Numbers} 
In this section, we will give an effective version of strong law of large numbers.

Let $(X,\, \Sigma,\,\mu)$ to be a probability space. For any $k\in\mathbb{N}$, let $(F_k(x))$ be sequence of $\mu$-integrable identically distributed random variables with mean $F$ and variance $\sigma^2>0$ on the probability measure space $(X,\Omega,\mu)$. The strong law of large numbers says that if all the $F_k$ are independent, then for $\mu$-almost every $x\in X$, \begin{align*}
    \lim_{N\to\infty}\frac{1}{N}\sum_{k=1}^NF_k(x)=F.
\end{align*}

In fact, the assumption that all $F_k$ are independent is stronger than needed for the conclusion to hold. In fact, if we have that   
\begin{align}\label{assumption_sloln}
    \int_X{\left(\sum_{k=m+1}^n(F_k(x)-F)\right)}^2\,\textrm{d}\mu\leq\sigma^2(n-m)\max{(1,F)},
    \end{align}
for any $m,n\in\mathbb{N}$ with $m<n$, then it follows from Lemma 1.5 from \cite{harman1998metric} that for almost every $x\in X$, as $N\to\infty$, 
\begin{align*}
    \frac{1}{N}\sum_{k=1}^NF_k(x)=F+O\left(N^{-1/2}\log^{2}{N}\right)\to F. 
\end{align*}
The following lemma shows that assumption \eqref{assumption_sloln} is indeed weaker than independence. In fact, we further show that the assumption that all the $F_{k}(x)$ are identical is unnecessary too.
\begin{lemma}
Suppose all $F_{k}(x)$ are independent, with finite means $F_{k}$ and variances $\sigma_{k}^{2}$. We assume there is a finite universal bound on the means $F_{k}$, and write $\Tilde{F}_{k}=\max\{F_{k},\,1\}$. Similarly, assume there is a finite universal bound on the $\sigma_{k}^{2}$ and write $\sigma^{2}=\max_{k}\{\sigma_{k}^{2},\,1\}$. Then for any $m,n\in\mathbb{N}$ with $m<n$,

\begin{align}\label{not indentical assump}
    \int_X{\left(\sum_{k=m+1}^n(F_k(x)-F_{k})\right)}^2\,\textrm{d}\mu\leq K\sum_{k=m+1}^{n}\Tilde{F}_{k},
    \end{align}
for a constant $K>0$
\end{lemma}
\begin{proof}
    For any $m,n\in\mathbb{N}$, if $m<n$ then, as the $F_{k}(x)$ are independent, we have that
    \begin{align*}
    \int_X{\left(\sum_{k=m+1}^n(F_k(x)-F_{k})\right)}^2\,\textrm{d}\mu&=\sum_{k=m+1}^n\int_X{(F_k(x)-F_{k})}^2\,\textrm{d}\mu\\
    &=\sum_{k=m+1}^{n}\sigma_{k}^{2}\\
    &\leq(n-m)\sigma^2\\
    &\leq\sigma^2\sum_{k=m+1}^{n}\Tilde{F}_{k}
. \end{align*}
Thus, the Lemma holds with $K=\sigma^{2}$. In the case that the $F_{k}(x)$ are identically distributed, from the final inequality we obtain \eqref{assumption_sloln}.
\end{proof}

Our effective version of strong law of large numbers is as follows:
\begin{thm} \hfill \\
    Let $(X,\, \Sigma,\,\mu)$ to be a probability space. For any $k\in\mathbb{N}$, let $(F_k(x))$ be sequence of $\mu$-integrable random variables with finite means $F_{k}$ and finite variances $\sigma_{k}^2>0$ on the probability measure space $(X,\Omega,\mu)$. We assume there is a finite universal bound for the variances, which we denote by $\sigma^{2}$, and assume there is a finite universal bound for the means $F_{k}$. Let $\Tilde{F}_{k}=\max\{F_{k},\,1\}$, so that $\sum_{k=1}^{\infty}\Tilde{F}_{k}$ diverges.
    
    Suppose that \eqref{not indentical assump} holds for any $m,n\in\mathbb{N}$ with $m<n$. Then for any $\varepsilon>0$ and $\delta>0$, there exists some $\mu$-measurable $E_{\varepsilon,\delta}\subset X$ such that $\mu(E_{\varepsilon,\delta})<\delta$ and for any $x\in X\setminus E_{\varepsilon,\delta}$ and $N\in\mathbb{N}$, 
    \begin{align}\label{LLN ineq}
        \left|\frac{1}{N}\sum_{k=1}^N\left(F_k(x)-F_{k}\right)\right|\leq K_{\varepsilon,\delta}\left(\frac{\varPhi^{1/2}(N)\log^{3/2+\varepsilon}{(\varPhi(N)})}{N}+\frac{\varPhi_{0}}{N}\right), 
    \end{align}
    where $\varPhi(N)=\sum_{k=1}^{N}\Tilde{F}_{k}$ and \begin{align*} 
        K_{\varepsilon,\delta}&=\max \left\{\alpha,\,\beta \right\},\\
        N_{\varepsilon,\delta}&=\left\lceil\frac{r_{\varepsilon,\delta}}{\varPhi_0}-1\right\rceil, \\
        r_{\varepsilon,\delta}&=\left\lceil\left(\frac{2\sigma^2}{\varepsilon\delta}\right)^{1/\varepsilon}\right\rceil+1,
    \end{align*}
    with
    \[\alpha =\frac{N_{\varepsilon,\delta}}{\max{\left(\varPhi_0^{1/2}\log^{3/2+\varepsilon}(\varPhi_0+2)+F,1\right)}},\]
    and
    \[\beta =\frac{2}{\log^{3/2+\varepsilon/2}{2}}\left(1+\frac{1}{\sqrt{2}\log^{3/2+\varepsilon}{4}}\right)
        \left(\frac{\log{4}}{\log{3}}\right)^{3/2+\varepsilon},\]
        where
        \[\varPhi_{0}=\max_{k}\{\Tilde{F}_{k}\}.\]
\end{thm}
We note that we do not need an assumption about the random variables being identically and independently distributed.
\begin{proof}
    The proof is essentially a direct application of Theorem \ref{thm3}. Take $K=\sigma^2$, $C=1$ and for any $k\in\mathbb{N}$, $f_k(x)=F_k(x)$, $f_k=F_{k}$, $\varphi_{k}=\Tilde{F}_{k}$ and $\varPhi_{0}$ as defined above. We apply Theorem \ref{thm3}, and the results follow by dividing both sides in the inequality by $N$. Although $C=1$ may not be a universal bound of $F_k(x)$, it follows from the proof of Theorem \ref{thm3} that it suffices to assume that $C=1$.
\end{proof}

If all the $F_{k}(x)$'s are identically distributed, then $\Tilde{F}_{k}=F$ is the same for all $k$. Thus, $\varPhi(N)=NF$, so substituting this into \eqref{LLN ineq}, we obtain that
\begin{align*}
        \left|\frac{1}{N}\sum_{k=1}^N\left(F_k(x)-F_{k}\right)\right|\leq K_{\varepsilon,\delta}\left(\frac{F^{1/2}\log^{3/2+\varepsilon}{(NF})}{N^{1/2}}+\frac{F}{N}\right). 
    \end{align*}

\section{Proof of the Probabilistic Results}
In this section we prove Theorems \ref{thm2} and \ref{thm3}.
\subsection{Proof of Theorem \ref{thm2}}
 Pick any $\varepsilon>0$. For any $N\in\mathbb{N}$ and $x\in X$, define 
 \begin{align*}
\Psi(N,x)&=\sum_{k=1}^Nf_k(x),\\
\Psi(N)&=\sum_{k=1}^Nf_k,\\
E(N,x)&=\Psi(N,x)-\Psi(N). 
\end{align*}
For any $j\in\mathbb{N}$, define 
\begin{align}
N_j&=\min\left\{n\in\mathbb{N}:\varPhi(n)>j^3\log^{1+\varepsilon}{(j+2)}\right\}, \label{Nj} \\
A_j&=\left\{x\in X:\left|E(N_j,x)\right|>j^2\log^{1+\varepsilon}{(j+2)}\right\}. \label{Aj}
\end{align}
we note that as $\Psi(N)\to\infty$, $N_j$ is well-defined for all $j\in\mathbb{N}$. Considering \eqref{1.2.7}, we see that
\begin{align}
    \mu\left(A_{j}\right)\min_{x \in A_{j}} \left(|E\left(N_{j},\,x\right)|^{2}\right) &\leq \int_{A_{j}}\left(\sum_{k=1}^{N_{j}}\left(f_{k}(x)-f_{k}\right)^{2}\right)\textrm{d}\mu(x) \nonumber \\
    &\leq \int_{X}\left(\sum_{k=1}^{N_{j}}\left(f_{k}(x)-f_{k}\right)^{2}\right)\textrm{d}\mu(x) \nonumber \\
    &\leq K\Phi(N_{j}).
\end{align}
It follows immediately from the above and \eqref{Aj} that
\begin{align}\label{abstract measure Aj}
    \mu\left(A_{j}\right) &\leq \frac{K\Phi(N_{j})}{\min_{x \in A_{j}}\left(|E\left(N_{j},\,x\right)|^{2}\right)} \nonumber \\
    &\leq \frac{K\Phi(N_{j})}{\left(j^{2}\log^{1+\varepsilon}(j+2)\right)^{2}} \nonumber \\
    &=\frac{K\Phi(N_{j})}{\left(j^{4}\log^{2+2\varepsilon}(j+2)\right)}.
\end{align}

For any $j\in\mathbb{N}$, $\log^{1+\varepsilon}{3}\leq j^3\log^{1+\varepsilon}{(j+2)}$. Further, by \eqref{first conditions on fk, phik} and \eqref{Nj} we see that
\[j^3\log^{1+\varepsilon}{(j+2)} < \varPhi(N_{j}) \leq j^3\log^{1+\varepsilon}{(j+2)}+1.\]
By applying the above and \eqref{Aj} to \eqref{abstract measure Aj}, we obtain that
\begin{align*}
\mu(A_j)&\leq\frac{Kj^3\log^{1+\varepsilon}{(j+2)}+K}{j^4\log^{2+2\varepsilon}{(j+2)}} \nonumber\\
&\leq\frac{(1+\log^{-1-\varepsilon}{3})Kj^3\log^{1+\varepsilon}{(j+2)}}{j^4\log^{2+2\varepsilon}{(j+2)}} \nonumber\\
&=\frac{1+\log^{-1-\varepsilon}{3}}{j\log^{1+\varepsilon}{(j+2)}}K.
\end{align*}
We now pick any $\delta>0$. Let 
\begin{align*}
j_{\varepsilon,\delta}&=1+\left\lceil\exp{\left(\frac{1+\log^{-1-\varepsilon}{3}}{\varepsilon\delta}K\right)^{1/\varepsilon}}\right\rceil.
\end{align*}
It follows that 
\begin{align}\label{sum of tail 1}
\sum_{j=j_{\varepsilon,\delta}}^\infty\mu(A_j)&<\int_{j_{\varepsilon,\delta}-1}^\infty\frac{(1+\log^{-1-\varepsilon}{3})K}{x\log^{1+\varepsilon}x}dx \nonumber \\
&=\frac{1+\log^{-1-\varepsilon}{3}}{\varepsilon\log^{\varepsilon}(j_{\varepsilon,\delta}-1)}K
\leq\delta.
\end{align}
Let 
\begin{align*}\label{E defn}
E_{\varepsilon,\delta}=\bigcup_{j=j_{\varepsilon,\delta}}^\infty A_j.
\end{align*}
Notice that, by sub-additivity of measures and \eqref{sum of tail 1}, \begin{align*}\mu(E_{\varepsilon,\delta})=\mu\left(\bigcup_{j=j_{\varepsilon,\delta}}^\infty A_j\right)\leq\sum_{j=j_{\varepsilon,\delta}}^\infty\mu(A_j)<\delta.
\end{align*}
After this initial set up, we can now give some lemmas we will use to complete the proof.

\begin{lemma} \label{Psi difference}
For any $j\in\mathbb{N}$, $j\geq2$, we have that 
\begin{align*}
\Psi(N_j)-\Psi(N_{j-1})\leq
\left(3+\frac{1+\varepsilon}{\log{3}}+\frac{1}{4\log^{1+\varepsilon}{4}}\right)j^2\log^{1+\varepsilon}{(j+2)},
\end{align*}
where $N_{j}$ is given at \eqref{Nj}.
\end{lemma}
\begin{proof}
Define $g:\mathbb{R}^+\to\mathbb{R}$ by  
\begin{align*}
g(j)=j^3\log^{1+\varepsilon}(j+2).
\end{align*}
Pick any $j\in\mathbb{N}$ such that $j\geq2$. By the Mean Value Theorem, there exists $\xi_j\in(j-1,j)$ such that 
\begin{align*}
g(j)-g(j-1)&=g'(\xi_j) \nonumber\\
&=3\xi_j^2\log^{1+\varepsilon}{(\xi_j+2)}+\frac{\xi_j^3}{\xi_j+2}(1+\varepsilon)\log^{\varepsilon}{(\xi_j+2)}\nonumber\\
&\leq 3\xi_j^2\log^{1+\varepsilon}{(\xi_j+2)}+\frac{\xi_j^2}{\log{3}}(1+\varepsilon)\log^{1+\varepsilon}{(\xi_j+2)}\nonumber\\
&=\left(3+\frac{1+\varepsilon}{\log{3}}\right)\xi_j^2\log^{1+\varepsilon}{(\xi_j+2)}\nonumber\\
&<\left(3+\frac{1+\varepsilon}{\log{3}}\right)j^2\log^{1+\varepsilon}{(j+2)}.
\end{align*}
We note that by \eqref{Nj}, we have that
\[\varPhi\left(N_{j}-1\right) \leq j^{3} \log^{1+\varepsilon}\left(j+2\right) < \varPhi\left(N_{j}\right).\]
From this, \eqref{first conditions on fk, phik}, and the fact that $j \geq 2$, we have that

\[\varPhi(N_{j})-\Psi(N_{j}) \geq \varPhi(N_{j-1})-\Psi(N_{j-1}).\]
More explicitly, this is because
\begin{align*}
    \varPhi(N_{j})-\Psi(N_{j})&=\left(\varPhi(N_{j-1})-\Psi(N_{j-1})\right)+\sum_{k=N_{j-1}+1}^{N_{j}}\left(\varphi_{k}-\psi_{k}\right),
\end{align*}
and by \eqref{first conditions on fk, phik} both terms on the right hand side are greater than equal to zero. Now, from the above and that $j\geq 2$ we see that
\begin{align*}
\Psi(N_j)-\Psi(N_{j-1})
&\leq\varPhi(N_j)-\varPhi(N_{j-1}) \nonumber \\
&=\varPhi(N_j-1)+\varphi_{N_j}-\varPhi(N_{j-1}) \nonumber\\
&\leq g(j)-g(j-1)+1 \nonumber \\
&<\left(3+\frac{1+\varepsilon}{\log{3}}\right)j^2\log^{1+\varepsilon}{(j+2)}+\frac{j^2\log^{1+\varepsilon}(j+2)}{4\log^{1+\varepsilon}{4}} \nonumber\\
&=\left(3+\frac{1+\varepsilon}{\log{3}}+\frac{1}{4\log^{1+\varepsilon}{4}}\right)j^2\log^{1+\varepsilon}{(j+2)},
\end{align*}
as stated in the lemma.
\end{proof}

\begin{lemma} \label{js in terms of Phi}
For any $\varepsilon>0$ and $j\in\mathbb{N}$, $j\geq2$, we have that
\begin{align*}
j^2\log^{1+\varepsilon}{(j+2)}\leq
\frac{4}{\log^{2\varepsilon/3}{(\varPhi_0+2)}}\left(\frac{\log{4}}{\log{3}}\right)^{1+\varepsilon}\varPhi^{2/3}(N_{j-1})\log^{1/3+\varepsilon}{(\varPhi{(N_{j-1})}+2)}.
\end{align*}
\end{lemma}
\begin{proof} 
Since $j\geq2$,
\begin{align*}
\frac{j^2\log^{1+\varepsilon}{(j+2)}}{(j-1)^{2}\log^{1+\varepsilon}{(j+1)}}
&\leq 4\left(\frac{\log{4}}{\log{3}}\right)^{1+\varepsilon}.
\end{align*}
Thus, considering \eqref{Nj}, we see that
\begin{align*}
j^2\log^{1+\varepsilon}{(j+2)}
&\leq4\left(\frac{\log{4}}{\log{3}}\right)^{1+\varepsilon}(j-1)^2\log^{1+\varepsilon}{(j+1)} \nonumber\\
&\leq4\left(\frac{\log{4}}{\log{3}}\right)^{1+\varepsilon}\varPhi^{2/3}(N_{j-1})\log^{(1+\varepsilon)/3}{(\varPhi(N_{j-1})+2)}\nonumber\\
&\leq\frac{4}{\log^{2\varepsilon/3}{(\varPhi_0+2)}}\left(\frac{\log{4}}{\log{3}}\right)^{1+\varepsilon}\varPhi^{2/3}(N_{j-1})\log^{1/3+\varepsilon}{(\varPhi{(N_{j-1})}+2)},
\end{align*}
as stated in the lemma.
\end{proof}

We are now in a position to complete the proof of Theorem \ref{thm2}. Pick any $x\in X\setminus E_{\varepsilon,\delta}$ and $N\in\mathbb{N}$ such that $N>N_{j_{\varepsilon,\delta}}$. Since by assumption,
\begin{align*}
\lim_{n\to\infty}\varPhi(n)=+\infty,
\end{align*}
there exists $j\in\mathbb{N}$ such that $N_{j-1}\leq N<N_j$ and $j>j_{\varepsilon,\delta}$. Hence $x\not\in A_j$. It then follows from \eqref{Aj}, Lemma \ref{Psi difference} and Lemma \ref{js in terms of Phi} that
\begin{align}\label{first thm last ineq.1}
\Psi(N,x)-\Psi(N) &\leq\Psi(N_j,x)-\Psi(N_{j-1}) \nonumber\\
&\leq\Psi(N_j,x)-\Psi(N_j)+\Psi(N_j)-\Psi(N_{j-1}) \nonumber \\
&\leq|E(N_j,x)|+\Psi(N_j)-\Psi(N_{j-1}) \nonumber\\
&\leq j^2\log^{1+\varepsilon}{(j+2)}+\left(3+\frac{1+\varepsilon}{\log{3}}+\frac{1}{4\log^{1+\varepsilon}{4}}\right)j^2\log^{1+\varepsilon}{(j+2)} \nonumber\\
&=\left(4+\frac{1+\varepsilon}{\log{3}}+\frac{1}{4\log^{1+\varepsilon}{4}}\right)j^2\log^{1+\varepsilon}{(j+2)} \nonumber\\
&\leq \frac{4}{\log^{2\varepsilon/3}{(\varPhi_0+2)}}\left(\frac{\log{4}}{\log{3}}\right)^{1+\varepsilon}\left(4+\frac{1+\varepsilon}{\log{3}}+\frac{1}{4\log^{1+\varepsilon}{4}}\right)\varPhi^{2/3}(N_{j-1})\log^{1/3+\varepsilon}{(\varPhi{(N_{j-1})}+2)}\nonumber\\
&\leq \frac{4}{\log^{2\varepsilon/3}{(\varPhi_0+2)}}\left(\frac{\log{4}}{\log{3}}\right)^{1+\varepsilon}\left(4+\frac{1+\varepsilon}{\log{3}}+\frac{1}{4\log^{1+\varepsilon}{4}}\right)\varPhi^{2/3}(N)\log^{1/3+\varepsilon}{(\varPhi{(N)}+2)}.
\end{align}
We now consider $N\leq N_{\varepsilon,\delta}$. We note that by the definition of $f_{k}(x)$ and $f_{k}$, we have that for all $N \in \mathbb{N}$,
\[-N \leq \sum_{k=1}^{N}\left(f_{k}(x)-f_{k}\right) \leq CN.\]
It follows that
\[\sum_{k=1}^{N}\left(f_{k}(x)-f_{k}\right) \leq N(C+1).\]
We note that if $C\geq 1$, then we have that
\begin{equation}\label{first thm last ineq.2}
\left|\sum_{k=1}^{N}\left(f_{k}(x)-f_{k}\right)\right|\leq NC.    
\end{equation}
We now assume that $C\geq 1$; indeed if it is not, then all the functions $f_{k}(x)$ are also bounded above by 1. Thus, for all $N \leq N_{\varepsilon,\delta}$ inequality \eqref{first thm last ineq.2} holds, and for all $N > N_{\varepsilon,\delta}$ we have inequality \eqref{first thm last ineq.1}. This proves the theorem.

\subsection{Proof of Theorem \ref{thm3}}
We now prove Theorem \ref{thm3}. As commented in \cite{harman1998metric}, we improve the spacing between $N_{j}$'s to improve the bound obtained in Theorem \ref{thm2}. We initially follow the proof of Lemma 1.5 from \cite{harman1998metric}, before quantifying the remainder term.

For any $j\in\mathbb{N}$, define 
\begin{align*}
n_j=\max{\{n\in\mathbb{N}:\varPhi(n)<j\}}.
\end{align*}
For all $j\in\mathbb{N}$, $j$ can be written in binary as 
\begin{align*}
j=\sum_{v=0}^{\lfloor\log_2{j}\rfloor}2^v b(j,v),
\end{align*}
where $b:\mathbb{N}\times\mathbb{N}_0\to\{0,1\}$.

Define $r_j=\lfloor\log_2{j}\rfloor$ and let
\begin{align*}
B_j=\left\{(i,s)\in\mathbb{N}_0\times\mathbb{N}_0:i=\sum_{v=s+1}^{r_j}2^{v-s}b(j,v),\;b(j,s)=1,\;0\leq s\leq r_j\right\}.
\end{align*}
Notice that $|B_j|\leq r_j+1$ and 
\begin{align}\label{partition}
\left(0,n_j\right]=\bigcup_{(i,s)\in B_j}\left(n_{i2^s},n_{(i+1)2^s}\right].
\end{align}
An concrete example of how this works is given in \cite{harman1998metric} for $j=37$.

Returning to the proof, we now define  for any $(i,s)\in\mathbb{N}_0\times\mathbb{N}_0$ and $x\in X$ a function $F:\mathbb{N}_0\times\mathbb{N}_0\times X\to\mathbb{R}$ as follows:
\begin{align*}
F(i,s,x)=\sum_{k=n_{i2^s}+1}^{n_{(i+1)2^s}}(f_k(x)-f_k).
\end{align*}
Then, by \eqref{partition}, for any $x\in X$, 
\begin{align}\label{proof56}
\sum_{k=1}^{n_j}(f_k(x)-f_k)=\sum_{(i,s)\in B_j}F(i,s,x).
\end{align}
For any $\varepsilon>0$, define 
\begin{align*}
M_{\varepsilon}&=\frac{\sqrt{2}}{\log^{3/2+\varepsilon/2}{2}}>1.
\end{align*}

We now give some lemmas we will use in the rest of the proof. 

\begin{lemma}\label{lemma6} 
For any $\varepsilon>0$ and $\delta>0$, there exists $E_{\varepsilon,\delta}\subset X$ and $r_{\varepsilon,\delta}\in\mathbb{N}$ such that $\mu(E_{\varepsilon,\delta})<\delta$, and for any $x\in X\setminus E_{\varepsilon,\delta}$ and $r\in\mathbb{N}$, if $r>r_{\varepsilon,\delta}$ then
\begin{align*}
\left|\sum_{k=1}^{n_r}(f_k(x)-f_k)\right|\leq M_{\varepsilon}\left(r^{1/2}\log^{3/2+\varepsilon}{(r+2)}\right).
\end{align*}
\end{lemma}
\begin{proof} 
By \eqref{proof56}, it suffices to prove that there exists $r_{\varepsilon,\delta}\in\mathbb{N}$ such that for any $r\in\mathbb{N}$ such that $r>r_{\varepsilon,\delta}$,  
\begin{align*}
\left|\sum_{(i,s)\in B_r}F(i,s,x)\right|\leq\sum_{(i,s)\in B_r}|F(i,s,x)|\leq M_{\varepsilon}\left(r^{1/2}\log^{3/2+\varepsilon}{(r+2)}\right).
\end{align*}

For any $r,\,i,\,s\in\mathbb{N}_0$, $x\in X$, define
\begin{align}
G(r,x)&=\sum_{s=0}^r\sum_{i=0}^{2^{r-s}-1}F^2(i,s,x), \label{G(r,x)}\\
\varPhi(i,s)&=\sum_{k=n_{i2^s}+1}^{n_{(i+1)2^s}}\varphi_k\coloneqq\varPhi{(n_{(i+1)2^s})} \nonumber \\
A_r&=\left\{x\in X:G(r,x)>r^{2+\varepsilon}2^r\right\} \label{Ar}.
\end{align}


Pick any $r\in\mathbb{N}$. By \eqref{1.2.7.a}, 
\begin{align*}\int_XG(r,x)\,d\mu(x)&\leq K\sum_{s=0}^r\sum_{i=0}^{2^{r-s}-1}\varPhi(i,s) \nonumber\\
&=K\sum_{s=0}^r\sum_{i=0}^{2^{r-s}-1}\left(\varPhi{(n_{(i+1)2^s})}-\varPhi{(n_{i2^s})}\right)\nonumber\\
&=K\sum_{s=0}^r\varPhi(n_{2^s})\nonumber\\
&\leq K(r+1)\varPhi(n_{2^r})\nonumber\\
&< 2Kr2^r.
\end{align*}
Since \begin{align}r^{2+\varepsilon}2^r\mu(A_r)&\leq\int_{A_r}G(r,x)\,d\mu(x)<2rK2^r, \nonumber
\end{align}
we obtain that 
\begin{equation*}\label{measure Ar}
  \mu(A_r)<2Kr^{-1-\varepsilon} . 
\end{equation*}

Pick any $\delta>0$ and define
\begin{align*}
r_{\varepsilon,\delta}&=1+\left\lceil\left(\frac{2K}{\varepsilon\delta}\right)^{1/\varepsilon}\right\rceil. 
\end{align*}
It follows that
\begin{align*}
\sum_{r=r_{\varepsilon,\delta}}^\infty\mu(A_r)&<\int_{r_{\varepsilon,\delta}-1}^\infty\frac{2K}{x^{1+\varepsilon}}dx =\frac{2K}{\varepsilon {(r_{\varepsilon,\delta}-1)}^\varepsilon}\leq\delta.
\end{align*}
We now take
\begin{align*}
E_{\varepsilon,\delta}=\bigcup_{r=r_{\varepsilon,\delta}}^\infty A_r.
\end{align*}
As in the previous proof, we note that by the sub-additivity of measures,
\begin{align*}
\mu(E_{\varepsilon,\delta})=\mu\left(\bigcup_{r=r_{\varepsilon,\delta}}^\infty A_r\right)\leq\sum_{r=r_{\varepsilon,\delta}}^\infty\mu(A_r)<\delta.
\end{align*}
Pick any $x\in X\setminus E_{\varepsilon,\delta}$. For any $r\in\mathbb{N}$, $r>r_{\varepsilon,\delta}$, we necessarily have that $x\not\in A_r$ and 
\begin{align}\label{G bound}
G(r,x)\leq r^{2+\varepsilon}2^r.
\end{align}
Pick any $j\in\mathbb{N}$ such that $j>2^{r_{\varepsilon,\delta}}$. Then $\lfloor\log_2{j}\rfloor+1>r_{\varepsilon,\delta}$. Taking $r= \lfloor\log_2{j}\rfloor+1>r_{\varepsilon,\delta}$, we find that
\begin{align*}
\sum_{(i,s)\in B_j}|F(i,s,x)| &=\sum_{s=0}^r\sum_{i=0}^{2^{r-s}-1}|F(i,s,x)|\chi_{B_j}(i,s),
\end{align*}
where $\chi_{B_j}(i,s)$ is the characteristic function on $B_{j}$.
Notice that $2^r\leq 2j$. Applying this, the Cauchy-Schwarz inequality, \eqref{G(r,x)}, \eqref{Ar} and \eqref{G bound} gives us that
\begin{align}
\sum_{(i,s)\in B_j}|F(i,s,x)|&\leq\left(\sum_{s=0}^r\sum_{i=0}^{2^{r-s}-1}|F(i,s,x)|^2\right)^{1/2}\left(\sum_{s=0}^r\sum_{i=0}^{2^{r-s}-1}\chi_{B_j}(i,s)\right)^{1/2} \nonumber\\
&=G^{1/2}(r,x)|B_j|^{1/2} \nonumber\\
&\leq r^{1+\varepsilon/2}2^{r/2}r^{1/2}\nonumber\\
&=2^{r/2}r^{3/2+\varepsilon/2}\nonumber\\
&\leq \sqrt{2}j^{1/2}\log^{3/2+\varepsilon/2}_2{j}\nonumber\\
&\leq \frac{\sqrt{2}}{\log^{3/2+\varepsilon/2}2}j^{1/2}\log^{3/2+\varepsilon/2}{j}\nonumber\\
&\leq M_{\varepsilon}\left(j^{1/2}\log^{3/2+\varepsilon/2}{j}\right), \nonumber
\end{align}
which proves the lemma.
\end{proof}

\begin{lemma}
For any $\varepsilon>0$ and $r\in\mathbb{N}$,
\begin{align*}
(r+1)^{1/2}\log^{3/2+\varepsilon}{(r+3)}\leq\sqrt{2}\left(\frac{\log{4}}{\log{3}}\right)^{3/2+\varepsilon}r^{1/2}\log^{3/2+\varepsilon}{(r+2)}.
\end{align*}
\end{lemma}\begin{proof}
This follows from the following bound:
\begin{align*}
\frac{(r+1)^{1/2}\log^{3/2+\varepsilon}{(r+3)}}{r^{1/2}\log^{3/2+\varepsilon}{(r+2)}}
&= \sqrt{\frac{r+1}{r}}\left(\frac{\log{(r+3)}}{\log{(r+2)}}\right)^{3/2+\varepsilon} \nonumber\\
&\leq\sqrt{2}\left(\frac{\log{4}}{\log{3}}\right)^{3/2+\varepsilon}.
\end{align*}
The lemma follows.
\end{proof}

We are now able to finish the proof of Theorem \ref{thm3}. Pick any $\varepsilon>0$ and $\delta>0$. By Lemma \ref{lemma6}, there exists a $\mu$-measurable $E_{\varepsilon,\delta}\subset X$ and $r_{\varepsilon,\delta}\in\mathbb{N}$ such that $\mu(E_{\varepsilon,\delta})<\delta$ and for any $x\in X \setminus E_{\varepsilon,\delta}$ and $r\in\mathbb{N}$ with $r>r_{\varepsilon,\delta}$, \begin{align*}
\left|\sum_{k=1}^{n_r}(f_k(x)-f_k)\right|&\leq M_{\varepsilon}\left(r^{1/2}\log^{3/2+\varepsilon}{(r+2)}\right).
\end{align*}
Pick any $n\in\mathbb{N}$ such that $n>n_{r_{\varepsilon,\delta}}$. Since by assumption
\begin{align*}
\lim_{N\to\infty}\varPhi(N)=\lim_{r\to\infty}n_r=+\infty,
\end{align*}
there exists $r\in\mathbb{N}$ such that  $n_r<n<n_{r+1}$ and $r>r_{\varepsilon,\delta}$. Notice that for any $x\in X$,
\begin{align*}
\sum_{k=1}^{n_r}f_k(x)\leq\sum_{k=1}^{n}f_k(x)\leq \sum_{k=1}^{n_{r+1}}f_k(x),
\end{align*}
and for any $x\in X\setminus E_{\varepsilon,\delta}$
\begin{align}
\left|\sum_{k=1}^{n_r}(f_k(x)-f_k)\right|&\leq M_{\varepsilon}\left(r^{1/2}\log^{3/2+\varepsilon}{(r+2)}\right), \label{bound a}
\end{align}
and
\begin{align}
\left|\sum_{k=1}^{n_{r+1}}(f_k(x)-f_k)\right|&\leq M_{\varepsilon}\left((r+1)^{1/2}\log^{3/2+\varepsilon}{(r+3)}\right) \label{bound b}.
\end{align}
We note that $M_{\varepsilon}>1$, $\varPhi(n_{r+1})<r+1$ and $r\leq\varPhi(n_r+1)$, so \begin{align*}
\sum_{k=n_r+1}^{n_{r+1}}f_k&=f_{n_r+1}+\sum_{k=n_r+2}^{n_{r+1}}f_k \nonumber\\
&\leq f_{n_r+1}+\sum_{k=n_r+2}^{n_{r+1}}\varphi_k \nonumber\\
&\leq \max_{1\leq k\leq n}f_k+\varPhi(n_{r+1})-\varPhi(n_{r}+1) \nonumber\\
&<\max_{1\leq k\leq n}f_k+(r+1)-r\nonumber\\
&=1+\max_{1\leq k\leq n}f_k.
\end{align*}
Since 
\begin{align*}
\sum_{k=1}^{n_r}f_k(x)-\sum_{k=1}^{n}f_k\leq\sum_{k=1}^{n}(f_k(x)-f_k)\leq\sum_{k=1}^{n_{r+1}}f_k(x)-\sum_{k=1}^{n}f_k,
\end{align*}
we get that
\begin{align*}
\sum_{k=1}^{n_r}(f_k(x)-f_k)-1-\max_{1\leq k\leq n}f_k&\leq\sum_{k=1}^{n}(f_k(x)-f_k)\nonumber \\
&\leq\sum_{k=1}^{n_{r+1}}(f_k(x)-f_k)+1+\max_{1\leq k\leq n}f_k.
\end{align*}
Notice that $r^{1/2}\log^{3/2+\varepsilon}{(r+2)}>1$. Hence by this, \eqref{bound a} and \eqref{bound b} we get that
\begin{align}
&\phantom{\leq}-M_{\varepsilon}\left(1+\frac{1}{\log^{3/2+\varepsilon}{3}}\right)\left( r^{1/2}\log^{3/2+\varepsilon}{(r+2)}+\max_{1\leq k\leq n}f_k\right) \nonumber\\
&\leq-M_{\varepsilon}\left(r^{1/2}\log^{3/2+\varepsilon}{(r+2)}\right)-1-\max_{1\leq k\leq n}f_k \nonumber\\
&\leq\sum_{k=1}^{n}(f_k(x)-f_k)\nonumber\\
&\leq M_{\varepsilon}\left((r+1)^{1/2}\log^{3/2+\varepsilon}{(r+3)}\right)+1+\max_{1\leq k\leq n}f_k \nonumber\\
&\leq M_{\varepsilon}\left(1+\frac{1}{\sqrt{2}\log^{3/2+\varepsilon}{4}}\right)\left((r+1)^{1/2}\log^{3/2+\varepsilon}{(r+3)}+\max_{1\leq k\leq n}f_k\right) \nonumber\\
&\leq \sqrt{2}M_{\varepsilon}\left(1+\frac{1}{\sqrt{2}\log^{3/2+\varepsilon}{4}}\right)
\left(\frac{\log{4}}{\log{3}}\right)^{3/2+\varepsilon}\left( r^{1/2}\log^{3/2+\varepsilon}{(r+2)}+\max_{1\leq k\leq n}f_k\right).
\end{align}
Thus, as $r\leq\varPhi(n)$, 
\begin{align}\label{second thm last ineq 1}
\left|\sum_{k=1}^{n}(f_k(x)-f_k)\right|
&\leq \sqrt{2}M_{\varepsilon}\left(1+\frac{1}{\sqrt{2}\log^{3/2+\varepsilon}{4}}\right)
\left(\frac{\log{4}}{\log{3}}\right)^{3/2+\varepsilon}\left( r^{1/2}\log^{3/2+\varepsilon}{(r+2)}+\max_{1\leq k\leq n}f_k\right) \nonumber\\
&\leq K_{\varepsilon,\delta}\left(\varPhi^{1/2}(n)\log^{3/2+\varepsilon}{(\varPhi(n)+2)}+\max_{1\leq k\leq n}f_k\right). 
\end{align}

We now consider $N \leq N_{\varepsilon,\delta}$. We have that for all $n \in \mathbb{N}$,
\[-N \max_{1\leq k\leq N}f_{k} \leq \sum_{k=1}^{N}\left(f_{k}(x)-f_{k}\right)\leq NC.\]

It follows immediately that
\begin{equation*}\label{second thm last ineq 2}
    \left|\sum_{k=1}^{N}\left(f_{k}(x)-f_{k}\right)\right| \leq N\left(C+\max_{1\leq k\leq N} f_{k}\right).
\end{equation*}
For all $N \leq N_{\varepsilon,\delta}$ we have that inequality \eqref{second thm last ineq 1} holds, and for all $N> N_{\varepsilon,\delta}$. This concludes the proof of Theorem \ref{thm3}.

\subsection{Proof of Lemma \ref{Aist general thm}}

\begin{proof}

The method to prove Theorem 1 from \cite{aistleitner2022metric} is generalised in this proof, and instead of requiring the specific measure space used in \cite{aistleitner2022metric}, this generalised version can be applied to general measure spaces.

Initially assume $C>4$. Define for any $k\in\mathbb{N}$,  \begin{align*}
    Q_{k}&=\min \left\{Q\in\mathbb{N}:\Psi(Q) \geq e^{k^{1/\sqrt{C}}}\right\}, 
\end{align*}
and let
\begin{align*}
    \mathcal{B}_{k}&=\left\{x \in X\,:\, \left|\sum_{1\leq q \leq Q_{k}}\left(f_{q}(x)-f_{q}\right)\right| \geq \frac{\Psi\left(Q_{k}\right)}{\left(\log \Psi\left(Q_{k}\right)\right)^{{C}/{4}}}\right\}. 
\end{align*}

By \eqref{Aist int} we obtain that
\begin{align}
    \mu\left(\mathcal{B}_{k}\right) \cdot \left(\frac{\Psi\left(Q_{k}\right)}{\left(\log \Psi\left(Q_{k}\right)\right)^{C/4}}\right)^{2} & \leq \int_{\mathcal{B}_{k}}\left(\sum_{1\leq q \leq Q_{k}}\left(f_{q}(x)-f_{q}\right)\right)^{2} \textrm{d}\mu \nonumber \\
    &\leq \int_{X}\left(\sum_{1\leq q \leq Q_{k}}\left(f_{q}(x)-f_{q}\right)\right)^{2} \textrm{d}\mu \nonumber \\
    &= O\left(\frac{\Psi(Q_{k})^{2}}{\left(\log \Psi(Q_{k})\right)^{C}}\right). \nonumber
\end{align}

It immediately follows that
\begin{equation}\label{measure Bk Aist}
    \mu\left(\mathcal{B}_{k}\right) \leq O\left(\left(\log \Psi\left(Q_{k}\right)\right)^{-C/2}\right) \leq O\left(k^{-\sqrt{C}/{2}}\right).
\end{equation}

As $C>4$ we have that
\[\sum_{k=1}^{\infty} \mu \left(\mathcal{B}_{k}\right)<+\infty,\]
so applying the Borel-Cantelli Lemma we find that almost all $x\in X$ are contained in at most finitely many of the sets $\mathcal{B}_{k}$. That is, for almost all $x$ there exists a $k_{0}(x)$ such that for all $k>k_{0}(x)$, we have that
\[\left|\sum_{1\leq q \leq Q_{k}}\left(f_{q}(x)-f_{q}\right)\right| \leq \frac{\Psi\left(Q_{k}\right)}{\left(\log \Psi\left(Q_{k}\right)\right)^{{C}/{4}}}.\]

For any $Q\geq 3$ there is a $k\in\mathbb{N}$ such that $Q_{k} \leq Q \leq Q_{k+1}$. This means that
\[\sum_{q=1}^{Q_{k}} \phi_{q} \leq\sum_{q=1}^{Q} \phi_{q} \leq \sum_{q=1}^{Q_{k+1}} \phi_{q}. \]
As by assumption $\phi_{q}\leq K$ for all $q\in\mathbb{N}$, we have that $\Psi\left(Q_{k}\right)\in \left[e^{k^{\frac{1}{\sqrt{C}}}},\,e^{k^{\frac{1}{\sqrt{C}}}}+K\right]$. It follows that
\[\frac{\Psi(Q_{k+1})}{\Psi(Q_{k})}=1+O\left(k^{-1+\frac{1}{\sqrt{C}}}\right)=1+O\left(\left(\log\Psi(Q_{k})\right)^{1-\sqrt{C}}\right).\]

From these formulae and the triangle inequality, it follows that for almost all $x \in X$ there exists a $Q_{0}\coloneqq Q_{0}(x)$ such that for all $Q>Q_{0}$ we have that
\[\left|\sum_{1\leq q \leq Q}\left(f_{q}(x)-f_{q}\right)\right|=O\left(\frac{\Psi(Q)}{\left(\log\Psi(Q)\right)^{\sqrt{C}-1}}\right).\]
As $C$ can be taken arbitrarily large, this proves the lemma.
\end{proof}

\subsection{Proof of Theorem \ref{effective v3}}
\begin{proof}
First we note that by \eqref{Aist int2}, for arbitrary $C>0$ there exists a $\mathcal{C} \in \mathbb{R}$ such that
\begin{equation*}\label{int with const}
\int_{X}\left(\sum_{1\leq q \leq Q}\left(f_{q}(x)-f_{q}\right)\right)^{2} \textrm{d}\mu \leq \mathcal{C}\frac{\Psi(Q)^{2}}{\left(\log \Psi(Q)\right)^{C}}. 
\end{equation*}

As before, for $C>4$ and any $k \in \mathbb{N}$, let
\begin{equation}\label{aist Qk defn}
   Q_{k}=\min \left\{Q\in\mathbb{N}\,:\,\Psi(Q) \geq e^{k^{\frac{1}{\sqrt{C}}}}\right\},\; k\geq 1, 
\end{equation}
and let
\begin{equation*}\label{aist Bk defn}
 \mathcal{B}_{k}=\left\{x \in X\,:\, \left|\sum_{1\leq q \leq Q_{k}}\left(f_{q}(x)-f_{q}\right)\right| \geq \frac{\Psi\left(Q_{k}\right)}{\left(\log \Psi\left(Q_{k}\right)\right)^{{C}/{4}}}\right\}.   
\end{equation*}

By \eqref{Aist int2} we obtain that
\begin{align}
    \mu\left(\mathcal{B}_{k}\right) \cdot \left(\frac{\Psi\left(Q_{k}\right)}{\left(\log \Psi\left(Q_{k}\right)\right)^{{C}/{4}}}\right)^{2} & \leq \int_{\mathcal{B}_{k}}\left(\sum_{1\leq q \leq Q_{k}}\left(f_{q}(x)-f_{q}\right)\right)^{2} \textrm{d}\mu \nonumber \\
    &\leq \int_{X}\left(\sum_{1\leq q \leq Q_{k}}\left(f_{q}(x)-f_{q}\right)\right)^{2} \textrm{d}\mu \nonumber \\
    &\leq \mathcal{C}\frac{\Psi(Q_{k})^{2}}{\left(\log \Psi(Q_{k})\right)^{C}}, \nonumber
\end{align}
so it follows that
\begin{equation}\label{measure Bk Aist2}
    \mu\left(\mathcal{B}_{k}\right) \leq \mathcal{C}\left(\log \Psi\left(Q_{k}\right)\right)^{-\frac{C}{2}} \leq \mathcal{C}k^{-\frac{\sqrt{C}}{2}},
\end{equation}
where the second inequality follows from \eqref{aist Qk defn}.

Let
\[k_{C,\,\delta} \coloneqq\min \left\{k\in\mathbb{N}\,:\,\mathcal{C}k^{-\frac{\sqrt{C}}{2}}\left(1+\frac{2k}{\sqrt{C}-2}\right)<\delta\right\}\]

Then, by the integral test for convergence and that $C>4$, it follows that
\begin{align*}\label{Aist tail estimate}
    \sum_{k=k_{C,\,\delta}}^{\infty}\mu(\mathcal{B}_{k}) &\leq \mu(\mathcal{B}_{k_{C,\,\delta}})+\int_{k_{C,\,\delta}}^{\infty}\mu(\mathcal{B}_{k})\textrm{d}k \nonumber \\
    &\leq \mathcal{C}\left(k_{C,\,\delta}^{-\frac{\sqrt{C}}{2}}+\int_{k_{C,\,\delta}}^{\infty}k^{-\frac{\sqrt{C}}{2}}\textrm{d}k \right) \nonumber \\
    &=\mathcal{C}\left(k_{C,\,\delta}^{-\frac{\sqrt{C}}{2}}+\frac{2}{2-\sqrt{C}}\left[k^{1-\frac{\sqrt{C}}{2}}\right]_{k_{C,\,\delta}}^{\infty}\right) \nonumber \\
    &=\mathcal{C}k_{C,\,\delta}^{-\frac{\sqrt{C}}{2}}\left(1+\frac{2k_{C,\,\delta}}{\sqrt{C}-2}\right) <\delta.
    \end{align*}
    
Define 
\[E_{C,\,\delta}= \bigcup_{k=k_{C,\,\delta}}^{\infty}\mathcal{B}_{k}.\]

Then by the above, $\mu \left(E_{C,\,\delta}\right)< \delta$.

We also note as before, as $C>4$ and by considering \eqref{measure Bk Aist2} we have that $\sum_{k=1}^{\infty}\mu(\mathcal{B}_{k})$ converges, so by the Borel-Cantelli lemma, for almost all $x \in X$, there exists a $k_{0}=k_{0}(x)$ such that for all $k>k_{0}$,
\[\left|\sum_{1\leq q \leq Q_{k}}\left(f_{q}(x)-f_{q}\right)\right| \leq \frac{\Psi\left(Q_{k}\right)}{\left(\log \Psi\left(Q_{k}\right)\right)^{{C}/{4}}}.\]

For any $Q \in \mathbb{N},\,Q\geq 3$, as before we have that there exists $k\in\mathbb{N}$ satisfying $Q_{k}\leq Q \leq Q_{k+1}$ and so
\[\sum_{q=1}^{Q_{k}}\phi_{q}\leq \sum_{q=1}^{Q}\phi_{q}\leq\sum_{q=1}^{Q_{k+1}}\phi_{q}.\]

By the definition of $Q_{k}$ and as $\phi_{q}\leq K$ for all $q$,
\[\Psi(Q_{k}) \in \left[e^{k^{\frac{1}{\sqrt{C}}}},\,e^{k^{\frac{1}{\sqrt{C}}}}+K\right].\]

Following methods from earlier in this paper, we prove the following lemma.

\begin{lemma}\label{g bound 2}
 For any $k \in \mathbb{N}$ we have that
 \begin{equation*}\label{Psi difference 2}
     0\leq\Psi\left(Q_{k+1}\right)-\Psi \left( Q_{k}\right) \leq \frac{e\Psi\left(Q_{k}\right) \left(\log \Psi\left(Q_{k}\right)\right)^{1-\sqrt{C}}}{\sqrt{C}}+K.
 \end{equation*}
\end{lemma}

\begin{proof}
We follow the same method we used to prove Lemma \ref{Psi difference}. Let
\[g(x) =e^{x^{1/\sqrt{C}}}.\]
Then by the mean value theorem, for any $k\in\mathbb{N}$, there exists a $\xi_{k}\in \left(k,\,k+1 \right)$ such that
\[g(k+1)-g(k)=g'(\xi_{k}).\]

Calculating, we see that for any $x \in (k,\,k+1)$,
\begin{align*}
    g'(x) &=\frac{e^{x^{\frac{1}{\sqrt{C}}}}x^{\frac{1}{\sqrt{C}}-1}}{\sqrt{C}} \nonumber \\
    &\leq \frac{e^{\left( k+1 \right)^{\frac{1}{\sqrt{C}}}}\left(k+1 \right)^{\frac{1}{\sqrt{C}}-1}}{\sqrt{C}} \nonumber \\
    &\leq \frac{e\cdot e^{k^{\frac{1}{\sqrt{C}}}} k^{\frac{1}{\sqrt{C}}-1}}{\sqrt{C}} \nonumber \\
    &\leq \frac{e\Psi\left(Q_{k}\right) k^{\frac{1}{\sqrt{C}}-1}}{\sqrt{C}},
\end{align*}
where the last line follows from the definition of $Q_{k}$.

Note that $\Psi\left(Q_{k+1}-1\right) < e^{\left(k+1\right)^{\frac{1}{\sqrt{C}}}}$. We now consider
\begin{align*}
    \Psi\left(Q_{k+1}\right) - \Psi \left(Q_{k}\right) &= \Psi\left(Q_{k+1}-1\right)+ \phi_{Q_{k+1}} -\Psi\left(Q_{k}\right) \nonumber \\
    & \leq e^{\left(k+1\right)^{\frac{1}{\sqrt{C}}}} - e^{k^{\frac{1}{\sqrt{C}}}} + K \nonumber \\
    &= g(k+1)-g(k)+K \nonumber \\
    &\leq \frac{e\Psi\left(Q_{k}\right) k^{\frac{1}{\sqrt{C}}-1}}{\sqrt{C}} + K \nonumber \\
    &\leq \frac{e\Psi\left(Q_{k}\right) \left(\log \Psi\left(Q_{k}\right)\right)^{1-\sqrt{C}}}{\sqrt{C}}+K,
\end{align*}
where the second to last line follows from Lemma \ref{g bound 2}, and the last line from the definition of $Q_{k}$ given at \eqref{aist Qk defn}.
\end{proof}

We now prove the result. Given $Q\in\mathbb{N}$ such that $Q>k_{C,\,\delta}$, with $Q_{k} \leq Q < Q_{k+1}$, we see that for any $x\in X\setminus E_{C,\,\delta}$,
\begin{align}\label{triangle ineq 2}
    \left|\sum_{q=1}^{Q}\left(f_{q}(x)-f_{q}\right)\right| &\leq \left|\sum_{q=1}^{Q_{k+1}}f_{q}(x) -\sum_{q=1}^{Q_{k}} f_{q}\right| \nonumber\\
    &\leq \left|\sum_{q=1}^{Q_{k+1}}f_{q}(x) -\Psi(Q_{k})\right| \nonumber \\
    &= \left|\sum_{q=1}^{Q_{k+1}}f_{q}(x) -\Psi\left(Q_{k+1}\right) +\Psi\left(Q_{k+1}\right) -\Psi(Q_{k})\right| \nonumber\\
    &\leq \left|\sum_{q=1}^{Q_{k+1}}f_{q}(x) -\Psi\left(Q_{k+1}\right)\right|+\left|\Psi\left(Q_{k+1}\right) -\Psi(Q_{k})\right| \nonumber\\
    &\leq \frac{\Psi\left(Q_{k+1}\right)}{\left(\log \Psi\left(Q_{k+1}\right)\right)^{{C}/{4}}}+\frac{e\Psi\left(Q_{k}\right) \left(\log \Psi\left(Q_{k}\right)\right)^{1-\sqrt{C}}}{\sqrt{C}}+K,
\end{align}
where the last line follows as $Q >k_{C,\,\delta}$ and from Lemma \ref{g bound 2}.

We note that
\begin{align*}
    \frac{\Psi\left(Q_{k+1}\right)}{\left(\log \Psi\left(Q_{k+1}\right)\right)^{{C}/{4}}}
    &\leq \frac{e^{\left(k+1\right)^{\frac{1}{\sqrt{C}}}}+K}{\left(\log \Psi\left(Q_{k}\right)\right)^{{C}/{4}}} \\
    &\leq \frac{e \cdot e^{k^{\frac{1}{\sqrt{C}}}}+K}{\left(\log \Psi\left(Q_{k}\right)\right)^{{C}/{4}}} \\
    &\leq \frac{e  \Psi\left(Q_{k}\right)+K}{\left(\log \Psi\left(Q_{k}\right)\right)^{{C}/{4}}}.
\end{align*}

Substituting this into \eqref{triangle ineq 2}, we obtain that
\begin{align*}
    \left|\sum_{q=1}^{Q}\left(f_{q}(x)-f_{q}\right)\right| &\leq \frac{e  \Psi\left(Q_{k}\right)+K}{\left(\log \Psi\left(Q_{k}\right)\right)^{{C}/{4}}} + \frac{e\Psi\left(Q_{k}\right) \left(\log \Psi\left(Q_{k}\right)\right)^{1-\sqrt{C}}}{\sqrt{C}}+K \\
    &\leq  \frac{e  \Psi\left(Q_{k}\right)+K}{\left(\log \Psi\left(Q_{k}\right)\right)^{\sqrt{C}-1}} + \frac{e\Psi\left(Q_{k}\right) }{\left(\log \Psi\left(Q_{k}\right)\right)^{\sqrt{C}-1}}+K \\
    &\leq 2\frac{e\Psi\left(Q_{k}\right) +K }{\left(\log \Psi\left(Q_{k}\right)\right)^{\sqrt{C}-1}}+K \nonumber \\
    &\leq 2\frac{e\Psi\left(Q\right) +K }{\left(\log \Psi\left(Q\right)\right)^{\sqrt{C}-1}}+K
\end{align*}

As $C$ can be chosen arbitrarily large, this means that
\begin{equation*}
    \left|\sum_{q=1}^{Q}\left(f_{q}(x)-f_{q}\right)\right| \leq 2\frac{e\Psi\left(Q\right) +K }{\left(\log \Psi\left(Q\right)\right)^{C}}+K;
\end{equation*}
that is,
\begin{equation}\label{effective A 1}
    \sum_{q=1}^{Q} f_{q}(x) \leq \sum_{q=1}^{Q}f_{q} + 2\frac{e\Psi\left(Q\right) +K }{\left(\log \Psi\left(Q\right)\right)^{C}}+K.
\end{equation}

This inequality holds for $Q>k_{C,\,\delta}$. We now consider when $Q\leq k_{C,\,\delta}.$ As $f_{q}(x)$ and $f_{q}$ are bounded above by $K$ for all $q$, we trivially have that
\begin{equation}\label{effective A 2}
-k_{C,\,\delta}K\leq-QK \leq \sum_{q=1}^{Q}\left(f_{q}(x)-f_{q}\right)\leq  QK \leq k_{C,\,\delta}K.    
\end{equation}

Combining \eqref{effective A 1} and \eqref{effective A 2}, we see that for all $Q \in \mathbb{N}$, we have that
\begin{equation*}
    \left|\sum_{q=1}^{Q} f_{q}(x) - \sum_{q=1}^{Q}f_{q}\right|\leq\max\left\{ k_{C,\,\delta}K,\,2\frac{e\Psi\left(Q\right) +K }{\left(\log \Psi\left(Q\right)\right)^{C}}+K\right\}.
\end{equation*}
\end{proof}

\bibliographystyle{siam}
\bibliography{name}

\begin{thebibliography}{1}

\bibitem{aistleitner2022metric}
{\sc C.~Aistleitner, B.~Borda, and M.~Hauke}, {\em On the metric theory of
  approximations by reduced fractions: a quantitative {Koukoulopoulos-Maynard}
  theorem}, 2022.

\bibitem{beresnevich2020number}
{\sc V.~Beresnevich and S.~Velani}, {\em Number theory meets wireless
  communications: an introduction for dummies like us}, in Number theory meets
  wireless communications, Springer, 2020, pp.~1--67.

\bibitem{Cassels_1950}
{\sc J.~W.~S. Cassels}, {\em Some metrical theorems in diophantine
  approximation. i}, Mathematical Proceedings of the Cambridge Philosophical
  Society, 46 (1950), p.~209–218.

\bibitem{duffin1941khintchine}
{\sc R.~J. Duffin and A.~C. Schaeffer}, {\em Khintchine’s problem in metric
  {Diophantine} approximation}, Duke Mathematical Journal, 8 (1941),
  pp.~243--255.

\bibitem{harman1998metric}
{\sc G.~Harman and S.~Harman}, {\em Metric Number Theory}, London Mathematical
  Society monographs, Clarendon Press, 1998.

\bibitem{koukoulopoulos2020duffin}
{\sc D.~Koukoulopoulos and J.~Maynard}, {\em On the {Duffin-Schaeffer}
  conjecture}, Annals of Mathematics, 192 (2020), pp.~251--307.

\bibitem{Pollington2022}
{\sc A.~D. Pollington, S.~Velani, A.~Zafeiropoulos, and E.~Zorin}, {\em
  Inhomogeneous {Diophantine} approximation on $m_0$-sets with restricted
  denominators}, International Mathematics Research Notices, 2022 (2022),
  pp.~8571--8643.

\end{thebibliography}

\section{Appendix}

\subsection{Proofs of Lemmas \ref{E.Lemma5} to \ref{E.Lemma8}}

In this section we will give the proofs of results from Section 3.2 which are omitted there. For the sake of simplicity, each of the proofs is written as an expansion the original proof in \cite{Pollington2022}; we refer heavily to the original paper throughout.
 
Before beginning the proofs, we make a slight remark that, as commented in \cite{Pollington2022}, the equality part for identity (3.2.5 in \cite{harman1998metric}) does not always hold. In general, we have the following result instead:
\begin{lemma}
    Let $f:\mathbb{N}\to[0,1/2]$ be a function, $(a_j)_j$, $(b_j)_j$ be sequences of integers and \begin{align*}
        \mathcal{E}_j=\{x\in[0,1]:\|a_jx+b_j\|<f(j)\}.
    \end{align*}
    Then for any $j,k\in\mathbb{N}$, the Lebesgue measure $\lambda$ of $\mathcal{E}_j\cap \mathcal{E}_k$ is upper bounded by
    \begin{align*}
        \lambda(\mathcal{E}_j\cap \mathcal{E}_k)
        \leq 4f(j)f(k)+2\gcd{(a_j,a_k)}\min{\left(\frac{f(j)}{a_j},\frac{f(k)}{a_k}\right)}.
    \end{align*}
\end{lemma}
\begin{proof}
    Let $\mathcal{F}_i=\{x\in[0,1]:\|x+b_i\|<f(i)\}$, $c_j=a_j/\gcd{(a_j,a_k)}$ and $c_k=a_k/\gcd{(a_j,a_k)}$. Since $\mathcal{F}_k$ is an open interval, for any $y\in\mathbb{R}$, the number of integers in the translation $c_j\mathcal{F}_k-y$ is at most 
    \begin{align*}
        \lambda(c_j\mathcal{F}_k)+1
        \leq c_j\lambda(\mathcal{F}_k)+1
        \leq 2c_jf(k)+1.
    \end{align*}
    By Lemma 3.1 from \cite{harman1998metric} and the argument contained there, 
    \begin{align*}
        \lambda(\mathcal{E}_j\cap \mathcal{E}_k)
        &\leq\frac{1}{c_jc_k}\lambda(c_j\mathcal{F}_k)\left(\lambda(c_k\mathcal{F}_j)+1\right)
        =\frac{1}{c_jc_k}(2c_kf(j))(2c_jf(k)+1) \\
        &=4f(j)f(k)+2\gcd{(a_j,a_k)}\frac{f(j)}{a_j},
    \end{align*}
    and the result follows by interchanging the indices $j$ and $k$ for the last inequality.
\end{proof}
Hence, (3.2.5 in \cite{harman1998metric}) should be applied in this context by setting $\mathcal{E}_j=E_{q_j}^\gamma$, $a_j=q_j$, $f(j)=\psi(q_j)$, correcting the equals sign to $\leq$, and multiplying the second term by a factor of 2 to obtain
\begin{align*}
    \lambda(E_{q_m}^\gamma\cap E_{q_n}^\gamma)\leq 4\psi(q_m)\psi(q_n)+2\gcd{(q_m,q_n)}\min{\left(\frac{\psi(q_m)}{q_m},\frac{\psi(q_n)}{q_n}\right)},
\end{align*}
Thus, by (105 in \cite{Pollington2022}), we can modify (88 in \cite{Pollington2022}) to see that
\begin{align}
    \hat{W}^+_{m,n}(0)
    &\leq 4(1+\varepsilon_m)(1+\varepsilon_n)\psi(q_m)\psi(q_n)+2(1+\varepsilon_m)(1+\varepsilon_n)\gcd{(q_m,q_n)}\min{\left(\frac{\psi(q_m)}{q_m},\frac{\psi(q_n)}{q_n}\right)} \nonumber\\
    \label{eq88+}
    &\leq4\psi(q_m)\psi(q_n)+12\varepsilon_m\psi(q_m)\psi(q_n)+8\gcd{(q_m,q_n)}\min{\left(\frac{\psi(q_m)}{q_m},\frac{\psi(q_n)}{q_n}\right)},
\end{align}
as the sequence $(\varepsilon_n)_{n\in\mathbb{N}}$ is bounded above by 1.

We are now in a position to prove the lemmas \ref{E.Lemma5} to \ref{E.Lemma8}.

\begin{proof}[Proof of Lemma \ref{E.Lemma5}]
Pick any sequence $(\varepsilon_n)_{n\in\mathbb{N}}$ of real numbers in $(0,1]$. By our assumptions, following the original proof in \cite{Pollington2022}, we see that for any $a,b\in\mathbb{N}$, if $a<b$ then \begin{align}
    \label{eq69'}
    P(a,b)\coloneqq \left|\sum_{n=a}^b\mu(E_{q_n}^\gamma)-2\sum_{n=a}^b\psi(q_n)\right|\leq\sum_{n=a}^b\psi(q_n)\varepsilon_n+\sum_{n=a}^b\frac{3}{n^{A/B}{\varepsilon_n}^{1/2}}
. \end{align}

Recall that, by definition, for any $n\in\mathbb{N}$, $\Psi(n)=\sum_{k=1}^n\psi(n)$. For any $n\in\mathbb{N}$, let
\begin{align*}
    \varepsilon_n=\min{\left(1,\left(\Psi(n)\right)^{-2}\right)}.
\end{align*}
From \eqref{eq69'} and Lemma D4 from \cite{Pollington2022}, following the argument of the original proof, we obtain that 
\begin{align*}
    P(a,b)\leq 3+3\sum_{n=1}^\infty\frac{1}{n^{A/B-1}}=3+3\zeta(A/B-1).
\end{align*}

To finish the proof, we need to establish the ``other" upper bound. Set $\tau=A/B$ in \eqref{extra con L5}, and let $\varepsilon_n=2^{-2/3}\in(0,1)$. Then \eqref{eq69'} allows us to deduce that
\begin{align*}
    P(a,b)
    \leq\sum_{n=a}^b\left(\varepsilon_n+\frac{1}{n^{A/(2B)}{\varepsilon_n}^{1/2}}\right)\psi(q_n)
    \leq\sum_{n=a}^b\left(\varepsilon_n+\frac{1}{{\varepsilon_n}^{1/2}}\right)\psi(q_n)
    =\frac{3}{2^{2/3}}\sum_{n=a}^b\psi(q_n).
 \end{align*}
We note there is no factor of 3 in the above, as the $3$ on the right hand side of \eqref{eq69'} is cancelled by \eqref{extra con L5}.

The upper bound is optimal in this method by the choice of $(\varepsilon_n)_{n\in\mathbb{N}}$; the inequalities are true for any choice of $(\varepsilon_n)_{n\in \mathbb{N}}$, and the function $f_1:\mathbb{R}^+\to\mathbb{R}^+$ defined for $x>0$ by
\begin{align*}
    f_1(x)=x+\frac{1}{x^{1/2}}
, \end{align*}
has its global minimum of $3/2^{2/3}$ at $x=\varepsilon_n=2^{-2/3}$.
\end{proof}

\begin{proof}[Proof of Lemma \ref{E.Lemma6}] This proof is done by splitting $S(m,\,n)$ into various parts. Following from the original proof, it suffices to improve some estimates.
    
    By \eqref{eq8} and \eqref{eq32}, for any $n\in\mathbb{N}$ and $t\in\mathbb{Z}\setminus\{0\}$,
    \begin{align*}
        \left|\hat{\mu}(-tq_n)\right|\leq \nu C^{-A}n^{-A/B}.
    \end{align*}
    
    By (47 in \cite{Pollington2022}) and (53 in \cite{Pollington2022}),
    \begin{align*}
        |S_1(m,n)|&\leq 9\nu C^{-A} \frac{\psi(q_m)}{n^{A/B}{\varepsilon_n}^{1/2}} \\
        |S_2(m,n)|&\leq 9\nu C^{-A} \frac{\psi(q_n)}{m^{A/B}{\varepsilon_m}^{1/2}},
    \end{align*}
    where $ |S_1(m,n)|$ and $|S_2(m,n)|$ are given at the beginning of the proof of Lemma 6 from \cite{Pollington2022}.
    
    We now find an explicit bound for $S_{4}(m,\,n)$. Notice that, given the restriction $|sq_m-tq_n|\geq q_n/2$ imposed on $s,\,t \in \mathbb{Z} \setminus \{0\}$, by \eqref{eq32} and \eqref{eq8} we have that 
    \begin{align*}
        \left|\hat{\mu}(sq_m-tq_n)\right|
        &\leq \frac{\nu}{\log^A{|sq_m-tq_n|}}
        \leq \frac{\nu}{\log^A{(q_n/2)}} \\
        &\leq \frac{2^A\nu}{\log^A{q_n}}
        \leq2^A\nu C^{-A} n^{-A/B}
    ,\end{align*}
    where the second line follows because for any $n\in\mathbb{N}$, $q_n\geq4$, and from that we deduce that
     \begin{align*}
        \frac{\log{q_n}}{\log{(q_n/2)}}\leq2.
    \end{align*}
    We can now give an explicit upper bound for $S_4$ as definied in the original proof. We find that
    \begin{align*}
        |S_4(m,n)|
        \leq 9(2)^A\nu C^{-A} \frac{1}{n^{A/B}{\varepsilon_m}^{1/2}{\varepsilon_n}^{1/2}}
    .\end{align*}
    
    To bound $S_6$ as given in the original proof, note that for ${q_n}^\alpha\leq|sq_m-tq_{m}|<q_n/2$, by \eqref{eq32} and \eqref{eq8}, we have that
    \begin{align*}
        \left|\hat{\mu}(sq_m-tq_n)\right|
        \leq \frac{\nu}{\alpha^A\log^A{q_n}}
        \leq \nu C^{-A}\frac{1}{\alpha^A} m^{-A/B}
    .\end{align*}
    It follows from the argument of the original proof that
    \begin{align*}
        |S_6(m,n)|
        \leq 9\nu C^{-A}\frac{1}{\alpha^A}\frac{\psi(q_n)}{m^{A/B}{\varepsilon_m}^{1/2}}
    .\end{align*}
    
    Hence the result follows by adding these estimates together, as in the original proof.
\end{proof}

\begin{proof}[Proof of Lemma \ref{E.Lemma7}] 
Set
\begin{align*}
    D=D(m,n)
    =\frac{q_m}{q_n\psi(q_m){\varepsilon_m}^{1/2}}.
\end{align*}
Notice that, from the original proof and that as $\varepsilon_n\leq1$, using (51 in \cite{Pollington2022}) and (52 in \cite{Pollington2022}), we have that
\begin{align*}
    |T(m,n)|
    &\leq\sum_{t\in\mathbb{N}}\min{\left(\frac{4}{\pi^2}\frac{q_m^2}{q_n^2}\frac{1}{t^2\psi(q_m)\varepsilon_m},(2+\varepsilon_m)\psi(q_m)\right)}(2+\varepsilon_n)\psi(q_n) \\
    &\leq 9\sum_{t=1}^{\lfloor{D}\rfloor}\left(\psi(q_m)\psi(q_n)\right) +\frac{4}{\pi^2}\frac{{q_m}^2}{{q_n}^2}\frac{\psi(q_n)}{\psi(q_m)\varepsilon_m}\sum_{t=\lfloor{D}\rfloor+1}^{+\infty}\frac{1}{t^2}\\
    &\leq11\frac{q_m}{q_n}\frac{\psi(q_n)}{{\varepsilon_m}^{1/2}}\leq11\frac{q_m}{q_n}\frac{\psi(q_n)}{{\varepsilon_n}^{1/2}}.
\end{align*}
This follows as $(\varepsilon_n)_{n\in\mathbb{N}}$ is decreasing, and by applying the following  estimate:
\begin{align*}
    \sum_{t=\lfloor{D}\rfloor+1}^{+\infty}\frac{1}{t^2}
    \leq\frac{\pi^2}{6}\int_{D}^\infty\frac{dt}{t^2}=\frac{\pi^2}{6D}=\frac{\pi^2}{6}\frac{q_n\psi(q_m){\varepsilon_m}^{1/2}}{q_m},
\end{align*}

Thus, by \eqref{eq17}, we find that $q_m\leq{K_0}^{m-n}q_n$. By the formula for geometric sums, we find that 
\begin{align*}
    \sum_{m=1}^{n-1}\frac{q_m}{q_n}
    \leq\sum_{m=1}^{n-1}\frac{{K_0}^{m-n}q_n}{q_n}
    <   \sum_{i=1}^\infty{K_0}^{-i}
    =   \frac{1/K_0}{1-1/K_0}=\frac{1}{K_0-1}=K'
    .\end{align*}
\end{proof}

\begin{proof}[Proof of Lemma \ref{E.Lemma8}] Notice that, from the original proof, for any $m,n\in\mathbb{N}$, \begin{align*}
    |T(m,n)|
    &\leq \sum_{s>m^3/\psi(q_m)}\frac{1}{\pi^2s\psi(q_m)\varepsilon_m}(2+\varepsilon_n)\psi(q_n) 
    =\frac{3}{\pi^2}\frac{\psi(q_n)}{\psi(q_m)\varepsilon_m}\sum_{s>m^3/\psi(q_m)}\frac{1}{s^2} \leq\frac{1}{2}\frac{\psi(q_n)}{m^3\varepsilon_m}
    \leq\frac{\psi(q_n)}{m^2},
    \end{align*}
by reasoning akin to that in the proof of Lemma \ref{E.Lemma7}. Hence for any $a,b\in\mathbb{N}$, if $a<b$ then
    \begin{align*}
        \mathop{\sum\sum}_{a\leq m<n\leq b}|T(m,n)|
        &\leq\sum_{n=a+1}^{b}\left(\psi(q_n)\sum_{m=a}^{b-1}\frac{1}{m^2}\right)
        <\frac{\pi^2}{6}\sum_{n=a}^b\psi(q_n)
    .\end{align*}
We finish the proof by noting that $\pi^2/6<2$.
\end{proof}

Our next aim is to establish explicit estimates for Propositions 1 and 2 from \cite{Pollington2022}. Before we can do this, we first need to find explicit estimates for some inequalities. Take $\delta\in(0,1]$ and for any $n\in\mathbb{N}$,\begin{align*}
    \varepsilon_n\coloneqq\min\left\{\frac{1}{2^\delta},\frac{1}{(\sum_{k=a}^n\psi(q_k))^\delta}\right\},
\end{align*}
notice that 
\begin{equation}\label{en bound}
  {\varepsilon_n}^{-1}\leq\max\left\{2^\delta,n^\delta\right\}<2n.  
\end{equation}

In the case that $\sum_{k=a}^b\psi(q_k)>2$, inequality (102 in \cite{Pollington2022}) can be modified to \begin{align*}
    \mathop{\sum\sum}_{a\leq m<n\leq b}\frac{1}{n^{A/B}{\varepsilon_m}^{1/2}{\varepsilon_n}^{1/2}}\leq \zeta{\left(\frac{A}{B}\right)}\sum_{k=a}^b\psi(q_k)
.\end{align*}
For the case of $\sum_{k=a}^b\psi(q_k)\leq2$, the inequality (103 in \cite{Pollington2022}) can be modified into \begin{align*}
    \mathop{\sum\sum}_{a\leq m<n\leq b}\frac{1}{n^{A/B}{\varepsilon_m}^{1/2}{\varepsilon_n}^{1/2}}\leq2\zeta{\left(\frac{A}{2B}\right)}\sum_{k=a}^b\psi(q_k)
.\end{align*}

Since the Riemann Zeta function is decreasing on $\mathbb{R}^+$, and 
\begin{align*}
    \zeta{\left(\frac{A}{B}\right)}<\zeta{\left(\frac{A}{2B}\right)}<2\zeta{\left(\frac{A}{2B}\right)},
\end{align*}
in both cases we have that 
\begin{align}
    \label{both}
    \mathop{\sum\sum}_{a\leq m<n\leq b}\frac{1}{n^{A/B}{\varepsilon_m}^{1/2}{\varepsilon_n}^{1/2}}\leq2\zeta{\left(\frac{A}{2B}\right)}\sum_{k=a}^b\psi(q_k)
.\end{align}

By \eqref{en bound}, we have that
\begin{align*}
    \sum_{n=1}^\infty\frac{1}{n^{A/B}{\varepsilon_n}^{1/2}}<\sum_{n=1}^\infty\frac{\sqrt{2n}}{n^{A/B}}=\sqrt{2}\zeta{\left(\frac{A}{B}-\frac{1}{2}\right)}.
\end{align*}
Thus we can make the inequality before (104 in \cite{Pollington2022}) explicit by applying Lemma \ref{E.Lemma6} and \eqref{both} to obtain \begin{align*}
    \mathop{\sum\sum}_{a\leq m<n\leq b}|S(m,n)|\leq 9\nu C^{-A} \left(\sqrt{2}\zeta{\left(\frac{A}{B}-\frac{1}{2}\right)}\left(2+\alpha^{-A}\right)+2^{A+1}\zeta{\left(\frac{A}{2B}\right)}\right)\sum_{n=a}^b\psi(q_n)
    +\mathop{\sum\sum}_{a\leq m<n\leq b}|T(m,n)|.
\end{align*}

We now note that
\begin{align*}
    \left(\frac{3}{2}+\frac{\log{x}}{2\log{(3/2)}}\right)
    \leq\left(\frac{3}{2\log{2}}+\frac{1}{2\log{(3/2)}}\right)\log{x}
    \leq 4\log{x}.
\end{align*}
Thus, we make inequality (108 in \cite{Pollington2022}) explicit, showing that for any $\delta>0$ and $a,b\in\mathbb{N}$,
\begin{align*}
    \mathop{\sum\sum}_{a\leq m<n\leq b}\varepsilon_m\psi(q_m)\psi(q_n)\leq 4\left(\sum_{n=a}^b\psi(q_n)\right)^{2-\delta}\log{\left(\sum_{n=a}^b\psi(q_n)\right)}.
\end{align*}

Before proceeding, we remark there is a typo on (109 in \cite{Pollington2022}); the correct version should state:
\begin{align}\label{thing to bound}
    \mathop{\sum\sum}_{a\leq m<n\leq b}W^+_{m,n}(0)
    &\leq 2\left(\sum_{n=a}^b\psi(q_n)\right)^2
    +O\left(\sum_{n=a}^b\psi(q_n)\right)^{2-\delta}\log{\left(\sum_{n=a}^b\psi(q_n)\right)} \\
    &\qquad+O\left(\mathop{\sum\sum}_{a\leq m<n\leq b}\gcd{(q_m,q_n)}\min{\left(\frac{\psi(q_m)}{q_m},\frac{\psi(q_n)}{q_n}\right)}\right);
\end{align}
the coefficient of the main term in \cite{Pollington2022} is 4 instead of 2. Hence, by \eqref{eq88+} and bounds given above, \eqref{thing to bound} is explicitly given by 
\begin{align*}
    \mathop{\sum\sum}_{a\leq m<n\leq b}W^+_{m,n}(0)
    &\leq 2\left(\sum_{n=a}^b\psi(q_n)\right)^2+24\left(\sum_{n=a}^b\psi(q_n)\right)^{2-\delta}\log{\left(\sum_{n=a}^b\psi(q_n)\right)} \\
    &\qquad+4\mathop{\sum\sum}_{a\leq m<n\leq b}\gcd{(q_m,q_n)}\min{\left(\frac{\psi(q_m)}{q_m},\frac{\psi(q_n)}{q_n}\right)}.
\end{align*}

We are now in a position to make Propositions 1 and 2 from \cite{Pollington2022} effective. By applying Lemma \ref{E.Lemma7} with $\alpha=1/2$ for the formula for $T$, the implicit constants of inequality (110 in \cite{Pollington2022}) are given by
\begin{align*}
    \mathop{\sum\sum}_{a\leq m<n\leq b}|T(m,n)|\leq 11K'\sum_{n=a}^b\frac{\psi(q_n)}{{\varepsilon_n}^{1/2}}\leq \frac{11K'}{\varepsilon_b^{1/2}}\sum_{n=a}^b\psi(q_n)\leq11K'\left(\sum_{n=a}^b\psi(q_n)\right)^{1+\delta/2},
\end{align*}
where $K'=1/(K_0-1)$, and as the sequence $(q_n)_{n\in\mathbb{N}}$ is lacunary,  inequality (111 in \cite{Pollington2022}) becomes 
\begin{align*}
    \mathop{\sum\sum}_{a\leq m<n\leq b}\gcd{(q_m,q_n)}\min{\left(\frac{\psi(q_m)}{q_m},\frac{\psi(q_n)}{q_n}\right)}\leq K'\sum_{n=a}^b\psi(q_n).
\end{align*}

We are now able to give explicit estimates for the inequalities on page 8620 in \cite{Pollington2022}. If the assumptions of Proposition 1 from \cite{Pollington2022} are satisfied, as inequality (104 in \cite{Pollington2022}) has been made effective, we see that 
\begin{align}
    \mathop{\sum\sum}_{a\leq m<n\leq b}\mu(E_m^\gamma\cap E_n^\gamma)\leq&
    2\left(\sum_{n=a}^b\psi(q_n)\right)^2+24\left(\sum_{n=a}^b\psi(q_n)\right)^{2-\delta}\log^+{\left(\sum_{n=a}^b\psi(q_n)\right)}+11K'\left(\sum_{n=a}^b\psi(q_n)\right)^{1+\delta/2} \nonumber\\
    \label{Pollingtion 1}
    &+\left(4K'+9\nu C^{-A} \left(\sqrt{2}\zeta{\left(\frac{A}{B}-\frac{1}{2}\right)}\left(2+\alpha^{-A}\right)+2^{A+1}\zeta{\left(\frac{A}{2B}\right)}\right)\right)\sum_{n=a}^b\psi(q_n).
\end{align}

If instead the assumptions in Proposition 2 from \cite{Pollington2022} are satisfied, we similarly have that
\begin{align}
    \mathop{\sum\sum}_{a\leq m<n\leq b}\mu(E_m^\gamma\cap E_n^\gamma)\leq&
    2\left(\sum_{n=a}^b\psi(q_n)\right)^2
    +24\left(\sum_{n=a}^b\psi(q_n)\right)^{2-\delta}\log^+{\left(\sum_{n=a}^b\psi(q_n)\right)} \nonumber\\
   & +\left(2+9\nu C^{-A} \left(\sqrt{2}\zeta{\left(\frac{A}{B}-\frac{1}{2}\right)}\left(2+\alpha^{-A}\right)+2^{A+1}\zeta{\left(\frac{A}{2B}\right)}\right)\right)\sum_{n=a}^b\psi(q_n) 
   \nonumber\\
    \label{Pollingtion 2}
    &+ 4\mathop{\sum\sum}_{a\leq m<n\leq b}\gcd{(q_m,q_n)}\min{\left(\frac{\psi(q_m)}{q_m},\frac{\psi(q_n)}{q_n}\right)}
. \end{align}

We now obtain the explicit versions of Proposition 1 and 2 from \cite{Pollington2022} as stated in Section 3.2 by applying Lemma \ref{E.Lemma5}, setting $\delta=2/3$ and $\delta=1$ to \eqref{Pollingtion 1} and \eqref{Pollingtion 2} respectively. Notice that by Lemma \ref{E.Lemma5} with the stated assumptions, \begin{align*}
    4\left(\sum_{n=a}^b\psi(q_n)\right)^2\leq\left(\sum_{n=a}^b\mu(E_{q_n}^\gamma) \right)^2+4m_1(m_2+1)\sum_{n=a}^b\psi(q_n)
, \end{align*}
where $m_1$ and $m_2$ are given in \eqref{m1m2}.

\subsection{A More General Version of Lemmas 1.4 and 1.5 from \cite{harman1998metric}}

As noted previously, \cite{aistleitner2022metric} uses a theorem very much like Lemma 1.4 from \cite{harman1998metric}, but not quite the same. We give general versions of Lemmas 1.4 and 1.5 from \cite{harman1998metric}. 

We begin by generalising Lemma 1.4 from \cite{harman1998metric}.

\begin{thm}\label{Theorem 10 Harman}
Let $(X,\,\Sigma,\, \mu)$ be a measure space and suppose that $0<\mu(X) <+\infty$. Let $f_{i}(x),\, i \in \mathbb{N}$ be a sequence on non-negative $\mu$-measurable functions, where for all $i \in \mathbb{N},\, x \in X$ we have that $f_{i}(x) <K$ for some constant $K \in \mathbb{R}$. Further, let $f_{i},\, \phi_{i} \in \mathbb{R}$ be sequences of real numbers such that for any $i \in \mathbb{N}$,
\begin{equation}
    0 \leq f_{i} \leq \phi_{i} \leq K.
\end{equation}

For any $N \in \mathbb{N}$ define
\[\Phi(N) = \sum_{i=1}^{N} \phi_{i},\]
and suppose that $\lim_{N \to \infty}\Phi(N) = + \infty$. Further, assume that for all $N \in \mathbb{N}$ we have that
\begin{equation}\label{int general 2}
    \int_{X}\left(\sum_{i=1}^{N}\left(f_{i}(x)-f_{i}\right)\right)^{2} \textrm{d}\mu = O(F(\Phi(N))),
\end{equation}
where $F:\mathbb{R} \to \mathbb{R}$ is an eventually strictly increasing function. Further, assume that $G:\mathbb{N}\to \mathbb{R}$ and $H:\mathbb{N} \to \mathbb{R}$ are eventually strictly increasing functions, and that
\begin{equation} \label{attempt 2 sum converge}
    \sum_{k=1}^{\infty} \frac{F(k)}{H(k-1)^{2}} < +\infty.
\end{equation}
We further define $I:\mathbb{N} \to \mathbb{R}$ such that
\begin{equation}\label{attempt 2 I defn}
I(k)\asymp\left(G(k+1)-G(k)\right),    
\end{equation}
and assume that $I$ is eventually strictly increasing. That is, there exist positive constants $c,\,C>0$ such that
\[c I(k) \leq G(k+1)-G(k) \leq C I(k),\]
with $I$ eventually strictly decreasing.
 Then
\begin{equation*}
    \sum_{i=1}^{N}f_{i}(x) = \sum_{i=1}^{N}f_{i}+O\left(H(G^{-1}(\Phi(N)))+I(G^{-1}(\Phi(N)))\right),
\end{equation*}
where we have assumed that $\Phi(N)$ is sufficiently large so that $G$ is strictly increasing on $(\Phi(N)-\varepsilon,\, +\infty)$ so the inverse makes sense on the restricted domain.

\end{thm}

We note that we are able to write the conclusion in terms of $\Phi(N)$ due to the invertibility of $G$ on the restricted range; indeed, we alternatively could write the conclusion as
\begin{equation*}
    \sum_{i=1}^{N}f_{i}(x) = \sum_{i=1}^{N}f_{i}+O\left(H(k-1)+I(k-1)\right),
\end{equation*}
where $k$ is defined so that
\[N_{k-1} \leq N < N_{k},\]
with
\[ N_{k}= \min \left\{N\in\mathbb{N}\,:\, \Phi(N) \geq G(k)\right\}.\]

In some cases, for example Lemma 1.4 from \cite{harman1998metric}, it is easier to argue for asymptotics in terms of $\Phi(N)$ from this result in terms of $k$; see the examples given below.

We now show the theorem above implies some results we have already seen. We first consider Lemma 1.4 from \cite{harman1998metric}. We define the functions as follows:
\begin{align*}
    F\Phi((N))&=\sum_{k=1}^{N}\phi_{k}=\Phi(N),\\
    G(k)&=k^{3}(\log 2k)^{1+\varepsilon}, \\
    H(k)&=k^{2}(\log 2k)^{1+\varepsilon}, \\
    I(k)&= k^{2}(\log 2k)^{1+\varepsilon}.
\end{align*}
We can check these functions satisfy the conditions in the theorem. We thus find that
\[\sum_{i=1}^{N}f_{i}(x) = \sum_{i=1}^{N}f_{i}+O\left(k^{2}(\log 2k)^{1+\varepsilon}\right),\]
where $k$ is defined as above so that $N_{k-1}\leq N <N_{k}$. Upon noting that
\[k^{2}(\log 2k)^{1+\varepsilon}=O\left(\Phi(N_{k-1})^{2/3}\log \left(\Phi(N)+2\right)\right)\]
(this is shown explicitly in Lemma \ref{js in terms of Phi}, or follows from the definitions of $G_{k}$m and $N_{K}$), we attain the result given.

We now consider the case of Lemma \ref{Aist general thm}. Define the functions as follows:
\begin{align*}
    F(\Phi(N))&=\frac{\sum_{k=1}^{N}\phi_{k}}{\left(\log \sum_{k=1}^{N} \phi_{k}\right)^{C}}=\frac{\Phi(N)}{\left(\log \Phi(N)\right)^{C}}, \\
    G(k)&=e^{k^{1/\sqrt{C}}}, \\
    H(k)&=\frac{\Phi(N_{k})}{\left(\Phi(N_{k})\right)^{C/4}}, \\
    I(k) &= \left(\log\Phi(N_{k})\right)^{-\sqrt{C}+1}.
\end{align*}
We note that $N_{k}$ is defined to be the smallest $N$ such that $\Phi(N)\geq G(k)$ and thus $H$ and $I$ are functions of $k$. Following the reasoning given in the proof of Lemma \ref{Aist general thm} then gives the result.

\begin{proof}[Proof of Theorem \ref{Theorem 10 Harman}]
Define 
\begin{equation}\label{attemt 2 Nk defn}
    N_{k}= \min \left\{N\,:\, \Phi(N) \geq G(k)\right\}.
\end{equation}
Further, let 
\[\mathcal{B}_{k}= \left\{x \in X\,:\,\left|\sum_{i=1}^{N_{k}}\left(f_{i}(x)-f_{i}\right)\right|\geq H(k-1)\right\}. \]
We find that by \eqref{int general 2},
\begin{align}
  \mu(\mathcal{B}_{k})\left(H(k-1)\right)^{2} & \leq \int_{\mathcal{B}_{k}}\left(\sum_{i=1}^{N_{k}}\left(f_{i}(x)-f_{i}\right)\right)^{2}\textrm{d}\mu \nonumber \\
  &\leq \int_{X}\left(\sum_{i=1}^{N_{k}}\left(f_{i}(x)-f_{i}\right)\right)^{2}\textrm{d}\mu \nonumber \\
  &=O\left(F(\Phi(N_{k}))\right)=O\left(F(k)\right)
\end{align}
as $N_{k}$ is defined in terms of $k$. It immediately follows that
\[\mu(\mathcal{B}_{k}) \leq O\left(\frac{F(k)}{H(k-1)^{2}}\right).\]

By \eqref{attempt 2 sum converge} we have that
\[\sum_{k=1}^{\infty} \frac{F(k)}{H(k-1)^{2}}\]
converges, so it follows that
\[\sum_{k=1}^{\infty}\mu(\mathcal{B}_{k}) \leq O\left(\sum_{k=1}^{\infty} \frac{F(k)}{H(k-1)^{2}}\right)\]
converges. Applying the Borel-Cantelli Lemma we see that almost all $x \in X$ belong to at most finitely many $\mathcal{B}_{k}$; that is, for almost all $x\in X$, there exists a $k(x)$ such that for all $k >k(x)$, we have that
\begin{equation}\label{attempt 2 H bound}
 \left|\sum_{i=1}^{N_{k}}\left(f_{i}(x)-f_{i}\right)\right|\leq H(k-1).   
\end{equation}

We now need to bound the value of $\Phi(N_{k})-\Phi(N_{k-1}).$ By \eqref{attemt 2 Nk defn} we note that
\begin{align}
    \Phi(N_{k})-\Phi(N_{k-1}) &=\Phi(N_{k}-1)+\phi_{N_{k}}-\Phi(N_{k-1}) \nonumber \\
    &\leq O\left(G(k)+K-G(k-1)\right)\nonumber \\
    &\leq O\left(I(k-1)\right)
\end{align}
where the last line follows from \eqref{attempt 2 I defn}.

It now follows that, for $x \in X$ such that $x\not\in\bigcup_k\mathcal{B}_k$ isn't in the exceptional set, for sufficiently large $N$ such that $N_{k-1} \leq N \leq N_{k}$ where $k-1>k(x)$ we have that
\begin{align}\label{general in k}
    \left|\sum_{i=1}^{N}\left(f_{i}(x)-f_{i}\right)\right| &\leq\left|\sum_{i=1}^{N_{k}}f_{i}(x)-\sum_{i=1}^{N_{k-1}}f_{i}\right| \nonumber \\
    &=\left|\sum_{i=1}^{N_{k}}f_{i}(x)-\Phi(k)+\Phi(k)-\Phi(k-1)\right| \nonumber \\
    &\leq \left|\sum_{i=1}^{N_{k}}f_{i}(x)-\Phi(k)\right|+\left|\Phi(k)-\Phi(k-1)\right| \nonumber \\
    &\leq O\left(H(k-1)+I(k-1)\right).
\end{align}
We note that this is the result given in the remark above. We now take advantage of the invertibility of $G$ on the restricted domain.

By definition,  $\Phi(N_{k-1}) \geq G(k-1)$. We recall that we have
\[\Phi(N_{k-1})  \leq \Phi(N) \leq \Phi(N_{k}),\]
as $N_{k-1} \leq N \leq N_{k}$. Assume $k$ is large enough that $G$ is strictly increasing, so the inverse of $G$ exists on the restricted domain we are considering. Then, as $G$ is strictly increasing, so is $G^{-1}$. Thus
\[G^{-1}\left(\Phi(N)\right) \geq G^{-1}\left(\Phi(N_{k-1})\right) \geq G^{-1}\left(G(k-1)\right)=k-1.\]
Now, as $H$ and $I$ are also eventually strictly increasing, assuming $k$ is large enough we can substitute the above into \eqref{general in k} to find that
\[ \left|\sum_{i=1}^{N}\left(f_{i}(x)-f_{i}\right)\right| \leq O\left(H(G^{-1}(\Phi(N)))+I(G^{-1}(\Phi(N)))\right).\]

\end{proof}

We now do the same for Lemma 1.5 from \cite{harman1998metric}.

\begin{thm}\label{theorem 11 harman}
    Let $X$ be a measure space with measure $\mu$ such that $0<\mu(X)<\infty$. Let $f_{k}(x),\,k=1,\,2,\,\dots$ be a sequence of non-negative, $\mu$-measurable functions, and let $f_{k},\,\varphi_{k}$ be sequences of real numbers such that
    \begin{equation}\label{general fkphik}
        0 \leq f_{k}\leq \varphi_{k}.
    \end{equation}
    Write 
    \[\Phi(N)=\sum_{k=1}^{N}\varphi_{k},\]
    and assume that $\Phi(N) \rightarrow \infty$ as $N \rightarrow \infty$.

    Suppose that for arbitrary integers $m,\,n,\, 1 \leq m <n$ we have that 
    \begin{equation}\label{integral}
        \int_{X}\left(\sum_{m\leq k <n}\left(f_{k}(x)-f_{k}\right)\right)^{2}\textrm{d}\mu =O\left(\Tilde{F}\left(\sum_{m \leq k <n}\varphi_{k}\right)\right),
    \end{equation}
    where $\Tilde{F}:\mathbb{R}\rightarrow\mathbb{R}^{+}$ is an increasing function. Further, let $\Tilde{G},\,\Tilde{H}$ and $\Tilde{I}:\mathbb{N}\rightarrow \mathbb{R}^{+}$ be increasing functions such that
    the number 
    \[ n_{j}=\max \left\{n:\,\Tilde{F}(\Phi(n))=\Tilde{F}\left(\sum_{k=1}^{n} \varphi_{k}\right) < \Tilde{G}(j),\, j \in \mathbb{N}\right\},\]
    is well defined, 
    \[\sum_{r=1}^{\infty}\frac{\Tilde{G}\left(2^{r}\right)}{r^{1+\varepsilon}\Tilde{I}(2^{r})\Tilde{H}(r)}\]
    converges, and that
    \[\Tilde{F}(\Phi(n_{j}) = O \left(\Tilde{F}(\Phi(n_{j+1})\right).\]
    Then

    \begin{align}\label{general 1.5 ineq}
        \sum_{k=1}^{N}f_{k}(x)=\sum_{k=1}^{N}f_{k}+O\left(AB+\max_{1 \leq k \leq N}f_{k}\right),
    \end{align}
        where 
        \begin{align*}
         A&=\left(\log \left( \Tilde{G}^{-1}\left(\Tilde{F}\left(\Phi\left(N\right)\right)\right)\right)+2\right)^{3/2+\varepsilon,}\\
         B&=\left(\Tilde{I}\left(\Tilde{G}^{-1}\left(\Tilde{F}\left(\Phi\left(N\right)\right)\right)\right)\Tilde{H}\left(\log \left(\Tilde{G}^{-1}\left(\Tilde{F}\left(\Phi\left(N\right)\right)\right)\right)\right)\right)^{1/2}.
        \end{align*}

\end{thm}
We note that, roughly speaking, the purpose of function $
\Tilde{F}$ is to let you change the growth of the asymptotic bound in \eqref{integral}, $\Tilde{G}$ is to let you control the rate that $n_{j}$ grows, and $\Tilde{H}$ and $\Tilde{I}$ are to let you control the size of the sets
\[\left\{x \in X:\,G(r,\,x)>r^{2+\varepsilon}\Tilde{I}\left({2^{r}}\right)\Tilde{H}(r)\right\},\]
which we will apply the Borel-Cantelli Lemma to in the proof.

We note that Lemma 1.5 follows from the above by setting the functions to the following:
\begin{align*}
    \Tilde{G}(j)=j \\
    \Tilde{F}(x)=x \\
    \Tilde{H}(j)=1 \\
    \Tilde{I}(j)=j.
\end{align*}
we then note that $\Tilde{G}^{-1}(j)=j$, and putting these into \eqref{general 1.5 ineq} gives us the result of Lemma 1.5 from \cite{harman1998metric} as expected.

\begin{proof}[Proof of Theorem \ref{theorem 11 harman}]
    Define the sequence $n_{1},\,n_{2},\,\dots$ by
    \begin{equation}\label{nj defn}
        n_{j}=\max \left\{n:\,\Tilde{F}(\Phi(n))=\Tilde{F}\left(\sum_{k=1}^{n} \varphi_{k}\right) < \Tilde{G}(j)\right\}.
    \end{equation}
    We note that the $n_{j}$ need not be distinct.

    Suppose that \eqref{general 1.5 ineq} holds for $N=n_{j}$ for all $j$. Then, if $n_{r}<n<n_{r+1}$, we have that
    \[\sum_{k=1}^{n_{r}}f_{k}(x) \leq \sum_{k=1}^{n}f_{k}(x) \leq \sum_{k=1}^{n_{r+1}}f_{k}(x), \]
    while
    \[\sum_{k=1}^{n_{r}}f_{k}(x)=\sum_{k=1}^{n_{r}}f_{k}+O\left(\left(\Tilde{I}(r)\Tilde{H}(r)\right)^{1/2}\left(\log(r+2)\right)^{3/2+\varepsilon}\right),\]
    and
    \[\sum_{k=1}^{n_{r+1}}f_{k}(x)=\sum_{k=1}^{n_{r+1}}f_{k}+O\left(\left(\Tilde{I}(r+1)\Tilde{H}(r+1)\right)^{1/2}\left(\log(r+3)\right)^{3/2+\varepsilon}\right).\]

    Note that by \eqref{general fkphik}, we have that
    \begin{align*}
       \sum_{n_{r}<k\leq n_{r+1}}f_{k} &\leq \max_{k \leq n_{r+1}} f_{k}+\Phi\left(n_{r+1}\right)-\Phi\left(n_{r}+1\right)  \\
       &\leq 1+\max_{k\leq n}f_{k}.
    \end{align*}
    Combining these results then gives us \eqref{general 1.5 ineq} under the assumptions mentioned.
    
    It thus remains to establish the result for $N=n_{j}$. Following the proof of Lemma 1.5 from \cite{harman1998metric}, we express the integer $j$ in binary scale as
    \begin{equation*}
        j=\sum_{0 \leq \upsilon \leq \log_{2}j}b(j,\,\upsilon)2^{\upsilon}.
    \end{equation*}

    We then let 
    \begin{equation*}
        B(j)= \left\{(i,\,s):\, i = \sum_{\upsilon = s+1}^{r}b(j,\,\upsilon) 2^{\upsilon-s},\,b(j,\,s)=1,\,0\leq s \leq r\right\},
    \end{equation*}
    where $r=r(j)=\left[\log_{2} j\right]$. We further define 
    \begin{equation*}
        F(i,\,s,\,x)=\sum_{u_{0}<k\leq u_{1}}\left(f_{k}(x)-f_{k}\right),
    \end{equation*}
where, for $t \in \{0,\,1\}$, we define 
\begin{equation}\label{u defn}
    u_{t}=u_{t}(i,\,s)=\max \left\{n>0:\,\Phi(n)<(i+t)2^{s}\right\},
\end{equation}
with the convention that $\max \emptyset=0$. This notation splits up $\left[1,\,n_{j}\right]$ into a suitably small number of blocks; that is,
\begin{equation*}
    \left(0,\,n_{j}\right] = \bigcup_{(i,\,s)\in B(j)}\left(u_{0},\,u_{1}\right],
\end{equation*} 
with $u_{0},\,u_{1}$ given by \eqref{u defn}. For further discussion of this, see the discussion and example below \eqref{partition}.

To complete the proof and establish the result for $N=n_{j}$, it remains to demonstrate that
\begin{equation*}
    \sum_{(i,\,s)}\lvert F(i,\,s,\,x) \rvert =O(\left(\left[\log_{2}j\right]+1\right)^{3/2+\varepsilon}\Tilde{I}\left(j+1\right)^{1/2}\Tilde{H}(\left[\log_{2}j\right]+1)^{1/2}.).
\end{equation*}

Set
\begin{equation*}
    G(r,\,x)=\sum_{\substack{0 \leq s \leq r \\ i < 2^{r-s}}}F^{2}(i,\,s,\,x)
\end{equation*}
and
\begin{equation*}
    \Phi(i,\,s)=\sum_{u_{0}<k\leq u_{1}}\varphi_{k},
\end{equation*}
    with $u_{t}$ given by \eqref{u defn}.
We now note that by \eqref{int with const} and \eqref{nj defn} we have that
\begin{align*}
    \int_{X}G(r,\,x) \textrm{d}\mu &= O\left(\sum_{\substack{0 \leq s \leq r \\ 0 \leq i <2^{r-s}}}\Tilde{F}\left(\Phi\left(i,\,s\right)\right)\right) \\
    &\leq O\left((r+1)\Tilde{F}\left(\Phi(n_{2^{r}})\right)\right) \\
    &\leq O\left(\left(r+1\right)\Tilde{G}\left(2^{r}\right)\right).
\end{align*}
It follows that
\begin{equation*}
    \mu \left\{x \in X:\,G(r,\,x)>r^{2+\varepsilon}\Tilde{I}\left({2^{r}}\right)\Tilde{H}(r)\right\}<O\left(\frac{\Tilde{G}\left(2^{r}\right)}{r^{1+\varepsilon}\Tilde{I}(2^{r})\Tilde{H}(r)}\right).
\end{equation*}
    By assumption, $\sum_{r=1}^{\infty}\frac{\Tilde{G}\left(2^{r}\right)}{r^{1+\varepsilon}\Tilde{I}(2^{r})\Tilde{H}(r)}$ converges. Thus, by the Borel-Canteli lemma, for almost all $x \in X$ we have that
    \begin{equation}\label{bounds on G}
        G(r,\,x) < r^{2+\varepsilon}\Tilde{I}(2^{r})\Tilde{H}(r),
    \end{equation}
for $r> r(x)$.

We now let $r=\left[\log_{2}j\right]+1$, and suppose $x$ belongs to the set for which \eqref{bounds on G} holds. Recall that $\lvert B(j)\rvert \leq r$. An application of the Cauchy-Schwarz inequality gives us that
\begin{align*}
    \sum_{(i,\,s) \in B(j)}\lvert F(i,\,s,\,x)\rvert &\leq \lvert B(j)\rvert^{1/2}G^{1/2}(r,\,x) \\
    &\leq r^{1/2}\left(r^{2+\varepsilon}\Tilde{I}(2^{r})\Tilde{H}(r)\right)^{1/2} \\
    &\leq r^{1/2}r^{1+\varepsilon}\left(\Tilde{I}(2^{r})\Tilde{H}(r)\right)^{1/2} \\
    &= r^{3/2+\varepsilon}\left(\Tilde{I}(2^{r})\Tilde{H}(r)\right)^{1/2}\\
    &\leq \left(\left[\log_{2}j\right]+1\right)^{3/2+\varepsilon}\Tilde{I}\left(j+1\right)^{1/2}\Tilde{H}(\left[\log_{2}j\right]+1)^{1/2}.
\end{align*}

Now we note that by \eqref{nj defn},
\[\Tilde{F}\left(\Phi(n_{j})\right) < \Tilde{G}(j)\leq \Tilde{F}\left(\Phi(n_{j+1})\right).\]

As $\Tilde{G}(j)$ and $\Tilde{F}(j)$ are strictly increasing, it follows that
\[j \leq \Tilde{G}^{-1}\left(\Tilde{F}\left(\Phi\left(n_{j+1}\right)\right)\right).\]

Now, as by assumption $\Tilde{F}(\Phi(n_{j}) = O \Tilde{F}(\Phi(n_{j+1})$, for $N=n_{j}$, we have established that
    \begin{equation*}
        \sum_{k=1}^{N}f_{k}(x)=\sum_{k=1}^{N}f_{k}+O\left(A'B'\right),
    \end{equation*}
    where
    \begin{align*}
        A'&=\left(\log \left( \Tilde{G}^{-1}\left(\Tilde{F}\left(\Phi\left(n_{j}\right)\right)\right)\right)+2\right)^{3/2+\varepsilon},\\
        B'&=\Tilde{I}\left(\Tilde{G}^{-1}\left(\Tilde{F}\left(\Phi\left(n_{j}\right)\right)\right)\right)^{1/2}\Tilde{H}\left(\log \left(\Tilde{G}^{-1}\left(\Tilde{F}\left(\Phi\left(n_{j}\right)\right)\right)\right)\right)^{1/2},
    \end{align*}
    and establishing this completes the proof.
\end{proof}

\end{document}